\documentclass{amsart}
\pdfoutput=1
\usepackage{color}
\usepackage{latexsym}
\usepackage{import}
\usepackage{amsmath}
\usepackage{amscd}
\usepackage{amssymb}
\usepackage{amsfonts}
\usepackage{amsthm}
\usepackage{enumerate}
\usepackage{graphicx}
\usepackage{caption}
\usepackage{subcaption}
\usepackage{cite}
\usepackage{hyperref}
\usepackage{mathrsfs}
\newcommand{\sphere}{\real/\integer}

\newcommand{\real}{\mathbb{R}}
\newcommand{\integer}{\mathbb{Z}}
\newcommand{\NN}{\mathbb{N}}
\newcommand{\KK}{\mathcal{K}(n, \varepsilon)}
\newcommand{\Kk}[1]{\mathcal{K}_{#1}(n, \varepsilon)}
\newcommand{\Kp}[1]{\mathcal{K}'_{#1}(n, \varepsilon)}
\newcommand{\nonneg}{\mathbb{N}}

\newtheorem{thm}{Theorem}[section]
\newtheorem{cor}[thm]{Corollary}
\newtheorem{conj}[thm]{Conjecture}
\newtheorem{lem}[thm]{Lemma}
\newtheorem{prop}[thm]{Proposition}
\newtheorem{ques}[thm]{Question}

\theoremstyle{defn}
\newtheorem{defn}[thm]{Definition}

\theoremstyle{remark}
\newtheorem{rem}[thm]{Remark}

\numberwithin{equation}{section}

\begin{document}

\author[H.~Wen]{Haomin Wen$^+$} \address{ Department of
  Mathematics, University of Pennsylvania, Philadelphia, PA 19104-6395
USA} \email{weh@math.upenn.edu} \thanks{Supported in part by NSF grant DMS 10-03679}
%\author{Haomin Wen}
\title{Simple Riemannian surfaces are scattering rigid}
\maketitle

\begin{abstract}
Scattering rigidity of a Riemannian manifold allows one to
tell the metric of a manifold with boundary by looking at the
directions of geodesics at the boundary. Lens rigidity allows one to
tell the metric of a manifold with boundary from the same information
plus the length of geodesics. There are a variety of results about
lens rigidity but very little is known for scattering rigidity. We will
discuss the subtle difference between these two types of rigidities
and prove that they are equivalent for two-dimensional simple manifolds
with boundaries. In particular, this implies that two-dimensional
simple manifolds (such as the flat disk) are scattering rigid since
they are lens/boundary rigid (Pestov--Uhlmann, 2005).
\end{abstract}

\section{Introduction}
\subsection{The invisible Eaton lens}
The invisible Eaton lens \cite{eaton-1952, kerker-1969, hannay-haeusser-1993} is a gradient-index (GRIN) lens that looks like the vacuum from the outside,
but has an infinite refractive index at the center.
The refractive index $n$ of the invisible Eaton lens is given by
\begin{align*}
    \sqrt{n} = \frac{1}{n r} + \sqrt{\frac{1}{n^2 r^2} - 1}.
\end{align*}
One can think the Eaton lens as a Riemannian manifold 
with the conformally flat metric $n^2 g_0$ on the unit disk.
The metric has a singularity at the center,
and the trajectories of light will be geodesics in that Riemannian manifold.

As can be seen from Figure \ref{fig:eaton},
the direction of each light ray when entering the lens is
the same as the direction of the light ray when leaving the lens.
Hence there is no refraction visible from the outside,
(even though each light ray makes a complete circuit inside the Eaton lens,)
and thus this Eaton lens is invisible.

\begin{figure}[h]
    \center
    \includegraphics[width=0.20 \textwidth]{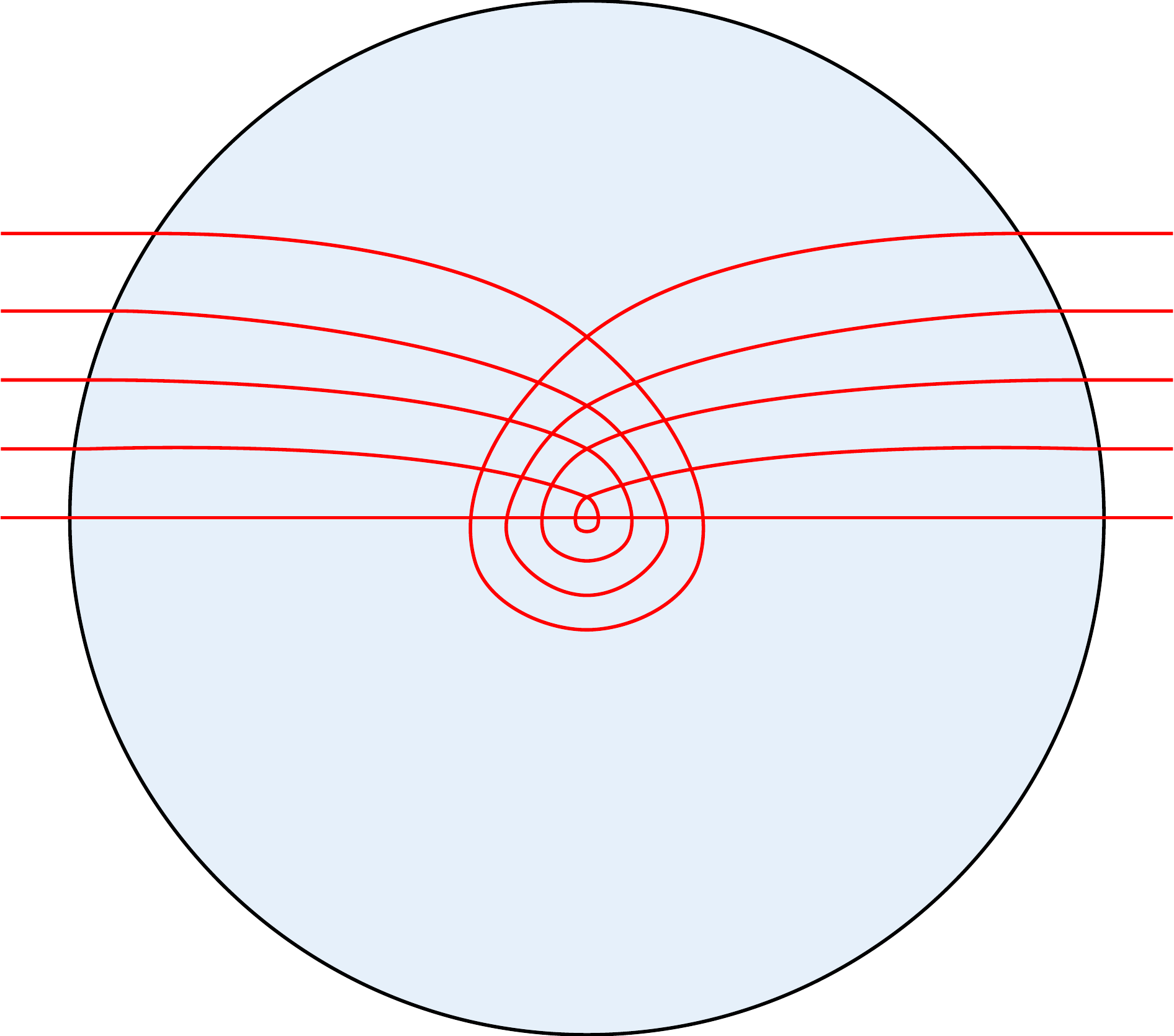}
    \caption{Trajectories of light in an invisible Eaton lens}
    \label{fig:eaton}
\end{figure}

\begin{ques}
  Can we have an invisible lens without singularities?
  \label{ques:inv}
\end{ques}

\subsection{Scattering rigidity and lens rigidity}
Question \ref{ques:inv} is equivalent to asking if flat balls are scattering rigid.
Simply put,
a Riemannian manifold $M$ is scattering rigid if $M$ is determined by its scattering data (see below) up to isometries which leave the boundary fixed.

Let $\pi : \Omega M \rightarrow M$ be the unit tangent bundle of $M$
and $\Omega_x M$ be the set of unit tangent vectors at $x$ for any $x \in M$.
Let $\partial \Omega M$ be the boundary of the unit tangent bundle of $M$.
In other words, $\partial \Omega M = \bigcup_{x \in \partial M} \Omega_x M$.
For each $x \in \partial M$,
let $\nu_M(x)$ be the unit normal vector of $M$ pointing inwards at $x$.
Then put
$\partial_+ \Omega_x M = \{X \in \Omega_x M: (X, \nu_M(x))_{g_M} > 0\}$,
$\partial_0 \Omega_x M = \{X \in \Omega_x M: (X, \nu_M(x))_{g_M} = 0\}$,
and $\partial_- \Omega_x M = \{X \in \Omega_x M: (X, \nu_M(x))_{g_M} < 0\}$.
Also,
write
$\partial_+ \Omega M = \bigcup_{x \in \partial M} \partial_+ \Omega_x M$, 
$\partial_0 \Omega M = \bigcup_{x \in \partial M} \partial_0 \Omega_x M$, 
and $\partial_- \Omega M = \bigcup_{x \in \partial M} \partial_- \Omega_x M$.

For each $X \in \partial_+ \Omega M$,
there is a geodesic $\gamma_X$ whose initial tangent vector is $X$.
Extend the geodesic as long as possible until it touches the boundary $\partial M$ again.
Let $\tau_X := \ell(\gamma_X)$, the length of $\gamma_X$.

If the geodesic $\gamma_X$ is of finite length,
call its tangent vector at the other end point $\alpha_M(X)$.
(See Figure \ref{fig:scattering_relation}.)
The map $\alpha_M : \partial_+ \Omega M \rightarrow \partial \Omega M$
defined above is called the \emph{scattering relation} of $M$.
Note that $\alpha_M(X)$ will be undefined if $\gamma_X$ is of infinite length.
\begin{figure}[h]
    \center
    \includegraphics[width=0.25\textwidth]{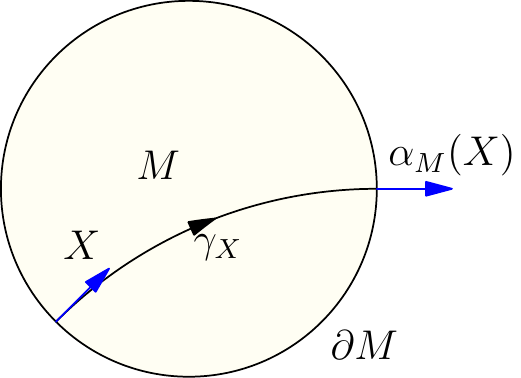}
    \caption{The scattering map $\alpha_M$}
    \label{fig:scattering_relation}
\end{figure}

Suppose that we have two Riemannian manifolds $(M, g_M)$, $(N, g_N)$
and an isometry $h : \partial M \rightarrow \partial N$ between their boundaries.
Then there is a natural bundle map $\varphi : \partial \Omega M \rightarrow \partial \Omega N$ defined as 
\begin{align}
  \varphi(a X + b \nu_M(x)) = a h_*(X) + b \nu_N(h(x))
  \label{eq:phi}
\end{align}
for any unit vector $X$ based at $x$ tangent to $\partial M$ and real numbers $a$ and $b$ such that $a^2 + b^2 = 1$.
$M$ and $N$ are said to have the same \emph{scattering data rel $h$} if $\varphi \circ \alpha_M = \alpha_N \circ \varphi$.
If we also have $\ell(\gamma_X) = \ell(\gamma_{\varphi(X)})$,
then we say $M$ and $N$ have the same \emph{lens data rel $h$}.
\begin{defn}
  We say a Riemannian manifold $M$ is \emph{scattering rigid} (resp. \emph{lens rigid}) if for any Riemannian manifold $N$ which has the same scattering data as $M$ (resp. lens data) rel $h$,
  (where $h : \partial M \rightarrow \partial N$ is an isometry,)
  we can always extend $h$ to an isometry from $M$ to $N$.
\end{defn}
We will omit ``rel $h$'' when $h$ is clear from the context or the specific choice of $h$ does not matter.
\begin{ques}[Equivalent to Question \ref{ques:inv}]
  Are flat balls scattering rigid?
\end{ques}
\begin{rem}
  Theorem \ref{thm:simple} (below) shows that 2-D flat disks are scattering rigid. Flat balls (of any dimensions) are known (Gromov \cite{gromov1983filling}) to be lens rigid.
\end{rem}
\subsection{Simple manifolds}
\begin{defn}
  A compact Riemannian manifold with boundary is \emph{simple}
  if
  \begin{enumerate}
    \item its boundary is strictly convex,
    \item there is a unique minimizing geodesic connecting any pair of points on the boundary,
    \item the manifold has no conjugate points.
  \end{enumerate}
\end{defn}
\begin{rem}
  Note that simple manifolds are topological balls.
\end{rem}
\begin{conj}[Michel \cite{michel-1981}]
  Simple Riemannian manifolds are lens (boundary) rigid.
  \label{conj:simple}
\end{conj}
\begin{thm}
  [Pestov--Uhlmann \cite{pestov-uhlmann-2005}]
    Simple Riemannian surfaces are lens (boundary) rigid.
    \label{thm:pu}
\end{thm}

\begin{rem}
  A Riemannian manifold is \emph{boundary rigid} if its metric is determined by the distance function between boundary points.
  The above statements are originally about boundary rigidity,
  which is equivalent to lens rigidity when the manifold is simple.
  Theorem \ref{thm:pu} confirms the conjecture for surfaces.
  There are a variety of results in higher dimensions (Besson--Courtois--Gallot \cite{MR1354289}, Burago--Ivanov \cite{burago2013area, MR2630062}, Croke--Kleiner \cite{MR1645381}, Michel \cite{michel-1981}),
  but it is still largely open.
\end{rem}

Our result extends Theorem \ref{thm:pu} to scattering rigidity.

\begin{thm}
    Simple Riemannian surfaces are scattering rigid.
    \label{thm:simple}
\end{thm}

\begin{rem}
  Simple Riemannian manifolds do not have trapped geodesics
  and trapped geodesics often make this type of rigidity problems much harder.
  Amazingly,
  the first (and the only one before this one) known result (Croke \cite{croke-2011}) of scattering rigidity
  is for the flat product metric on $\mathbb{S} \times D^n$,
  which has trapped geodesics.
\end{rem}

To get Theorem \ref{thm:simple} from Theorem \ref{thm:pu},
it suffices to show that $M$ and $N$ have the same lens data if they have the same scattering data, assuming that $M$ is simple.
Note that this is not true in general without the assumption that $M$ is simple. (See Figure \ref{fig:ct}.)
\begin{figure}[h]
    \centering
    \begin{subfigure}[b]{0.12\textwidth}
      \centering
      \includegraphics[width=\textwidth]{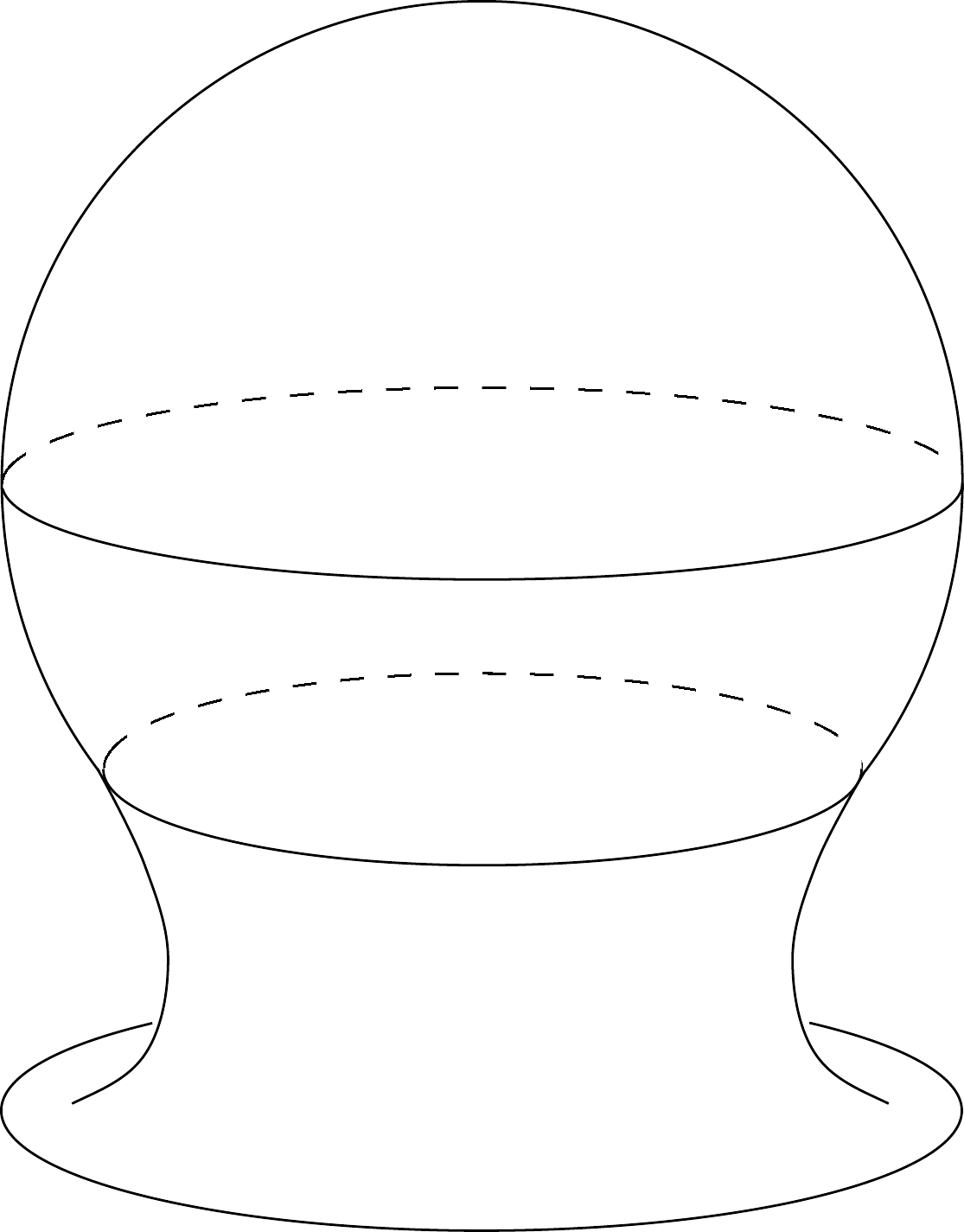}
      \caption{}
      \label{c1}
    \end{subfigure}%
    \quad
    \quad
    \quad
    \quad
    \begin{subfigure}[b]{0.12\textwidth}
      \centering
      \includegraphics[width=\textwidth]{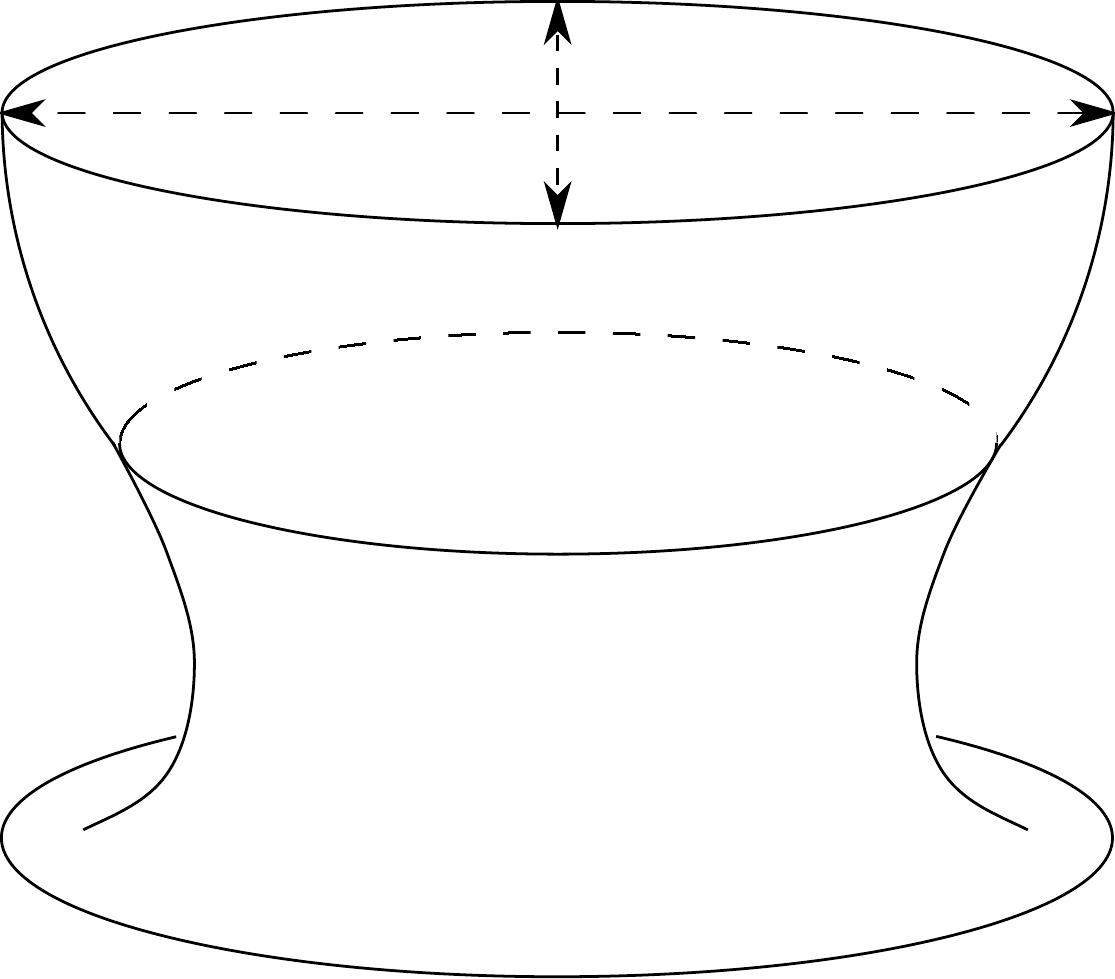}
      \caption{}
      \label{c2}
    \end{subfigure}
    \caption{\ref{c2} is obtained from \ref{c1} by removing the upper hemisphere and identifying antipodal points in the top boundary component. \ref{c1} and \ref{c2} have the same scattering data but different lens data.}
    \label{fig:ct}
\end{figure}

By the first variation of arc length, $\ell(\gamma_{\varphi(X)}) - \ell(\gamma_X)$ is equal to a constant $L \ge 0$.
If $L > 0$, then $\gamma_{\varphi(X)}$ converges to a closed geodesic of length $L$ as $X$ converges to a vector $X_0$ tangent to the boundary. (See Figure \ref{fig:closed}.) We will call this closed geodesic $\gamma_{X_0}$.
\begin{figure}[h]
  \centering
  \begin{subfigure}[b]{0.20\textwidth}
    \centering
    \includegraphics[width=\textwidth]{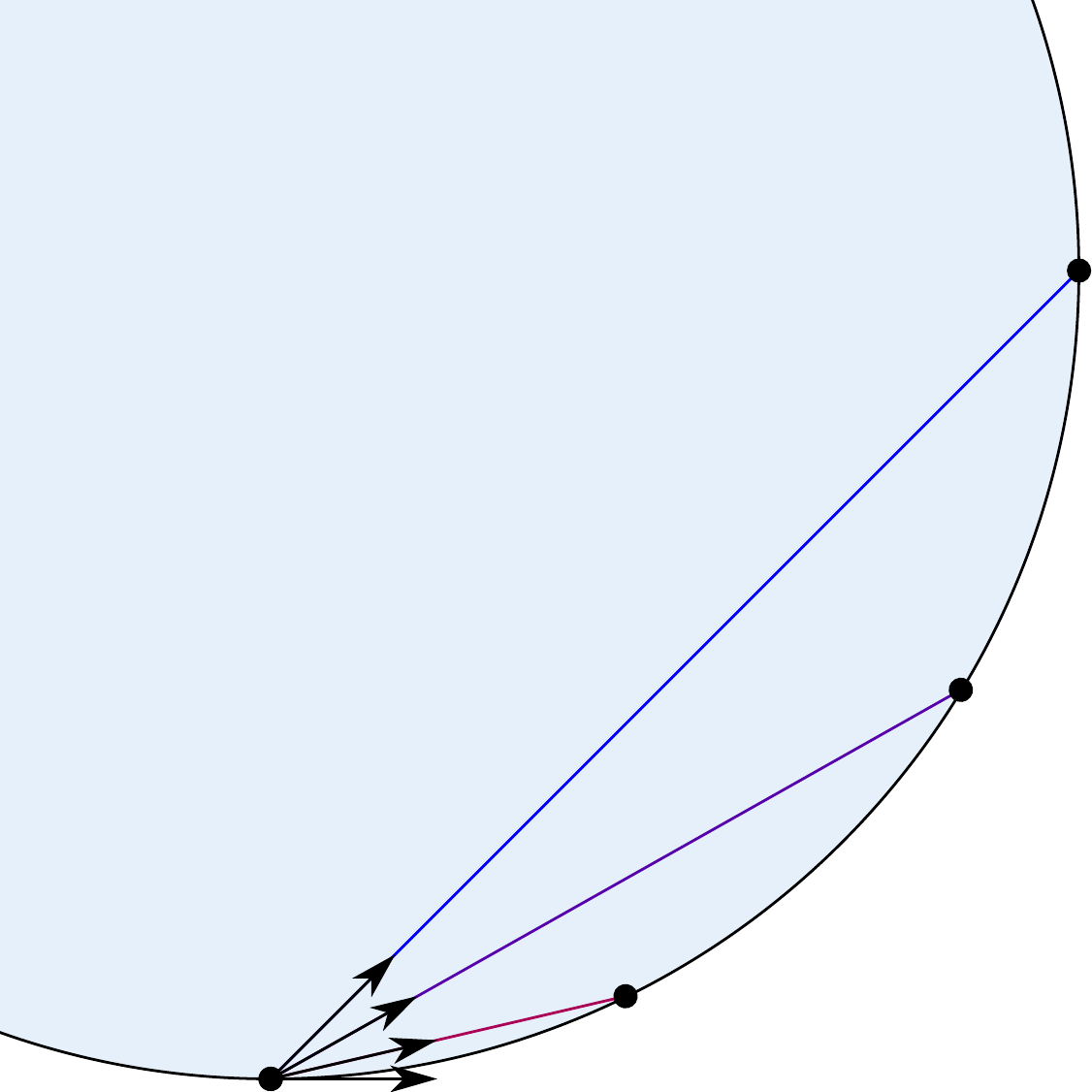}
    \caption{$\gamma_X$ in $M$}
  \end{subfigure}%
  \quad
  \begin{subfigure}[b]{0.20\textwidth}
    \centering
    \includegraphics[width=\textwidth]{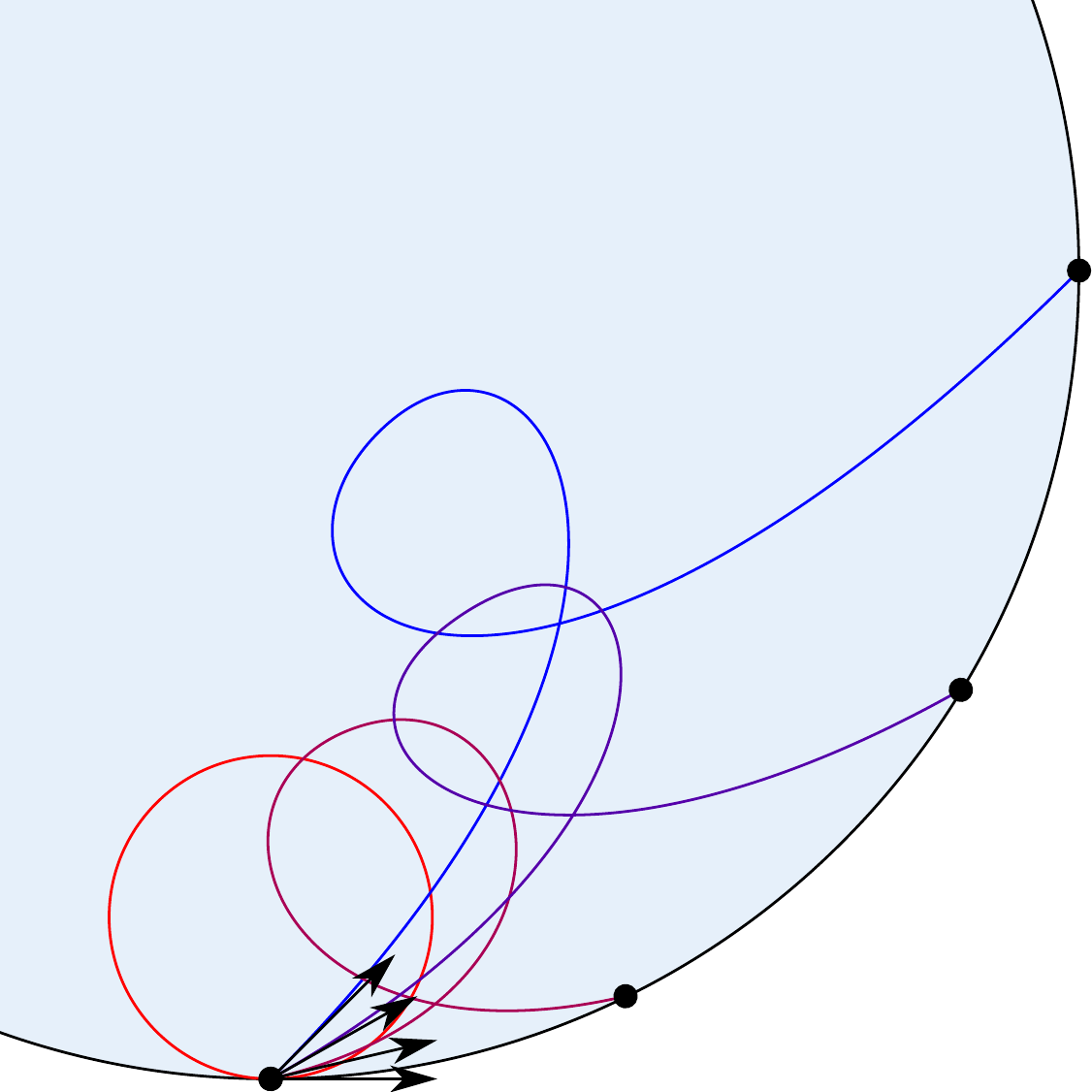}
    \caption{$\gamma_{\varphi(X)}$ in $N$}
  \end{subfigure}
  \caption{Closed geodesics?}
  \label{fig:closed}
\end{figure}

At first glance, this is very unlikely to happen,
since one expects $\partial N$ to be convex
as $\partial M$ is convex.
However,
the convexity of the boundary of a manifold,
being a local property,
is \textbf{not} determined by local scattering data
as illustrated by the invisible Eaton lens.
The boundary of the invisible Eaton lens is actually totally geodesic,
and we have closed geodesics running along the boundary.
The trickiest part of the proof is to get rid of these closed geodesics
using knot theory, (which only works in dimension $2$ so far).

\subsection{Scheme of the proof}
As explained in the previous section,
we need to close the gap between lens rigidity and scattering rigidity, that is, to show $L = 0$. Recall that $L = \ell(\gamma_{\varphi(X)}) - \ell(\gamma_X)$, the difference between the lengths of corresponding geodesics in $M$ and $N$, where $M$ and $N$ are two Riemannian manifolds with the same scattering data rel $h : \partial M \rightarrow \partial N$.

In section \ref{topology}, we will prove that $N$ is homeomorphic a disk.

Pick any $x \in \partial N$.
If $L > 0$,
then there is a closed geodesic $\gamma_x$ of length $L$ which is
tangent to $\partial N$ at $x$.
There are two such closed geodesics for each $x$,
but we can choose $\gamma_x$ properly such that
$\gamma_x$ moves continuously as $x$ moves.
In this section,
we will assume that $\gamma_x$ has multiplicity $1$.
The actual proof will be more complicated due to the possibility of higher multiplicities,
but the idea of the proof is the same.

The paper will study the isotopy type of the \emph{projectivized unit tangent vector field}
$P \circ \tilde\gamma_x : \sphere \rightarrow P \Omega N$ of $\gamma_x$ where
$\tilde\gamma_x : \sphere \rightarrow \Omega N$ is the unit tangent vector field of $\gamma_x$ (see \eqref{eq:unit_tangent} in section \ref{sec:knot}),
$P \Omega N = \Omega N / \{(x, \xi) \sim (x, -\xi)\}$ is the \emph{projectivized unit tangent bundle} of $N$,
and $P : \Omega N \rightarrow P \Omega N$ is the corresponding quotient map.

In section \ref{sec:knot}
we shall define a family of knot invariants
for contractible knots embedded in $P \Omega N$,
and then use those invariants to prove Theorem \ref{thm:knot},
which is interesting on its own.
\begin{thm}
    $P \circ \tilde{\gamma}$ is an isotopically non-trivial knot in $P \Omega N$
    for any smooth immersed curve $\gamma : \sphere \rightarrow N$ without self-tangencies.
    \label{thm:knot}
\end{thm}
\begin{rem}
  Theorem \ref{thm:knot} is purely knot-theoretical as it involves nether scattering data nor lens data. It is a bit surprising that this simple fact was not known before even for plane curves.
  Actually, it would be a completely different story if the projectivization were dropped: Chmutov--Goryunov--Murakami \cite{MR1776110} showed that every knot type in $\Omega \mathbb{R}^2$ (including the trivial type) is realized by the unit tangent vector field along an immersed plane curve.
\end{rem}

Notice that the union of $P \circ \tilde\gamma_x$ for all $x \in \partial N$ is
a torus immersed in $P \Omega N$.
We can perturb the immersion to an embedding.
Then we can prove that the torus is compressible by showing that $P \circ \tilde\gamma_x$
is contractible.
(Actually,
any embedded torus in $P \Omega N$ is compressible.)
Next, we can show that the other generator of the fundamental group of the torus
is not contractible in $P \circ \tilde\gamma_x$.
It follows that $P \circ \tilde\gamma_x$ bounds an embedded disk,
which contradicts Theorem \ref{thm:knot}.
Therefore, there is no such closed geodesics.

In the actual proof, we shall prove Theorem \ref{thm:simple} in section \ref{proof}
using a similar contradiction without the assumption on the multiplicity.
%\subsection{Acknowledgements}
%Many thanks to my advisor Christopher Croke for introducing me this subject
%and teaching me the techniques in this field.
%Many ideas in this paper stem from discussions with him.

\subsection{Acknowledgements}
Many thanks to my advisor Christopher Croke for introducing me this subject
and teaching me the techniques in this field.
Many ideas in this paper stem from discussions with him.
\section{Topology of $N$}
\label{topology}
Through out the paper (except in section \ref{sec:knot}), $M$ and $N$ will be two Riemannian surfaces with the same scattering data rel $h : \partial M \rightarrow \partial N$
where $h$ is an isometry.
Also, $M$ is assumed to be simple.
$\varphi : \partial \Omega M \rightarrow \partial \Omega N$
is the induced bundle map defined in \eqref{eq:phi}.
We aim to prove the following result in this section

\begin{prop}
  $N$ is homeomorphic to a 2-disk if $M$ is simple.
  \label{prop:disk}
\end{prop}

If $L := \ell(\gamma_{\varphi(X)}) - \ell(\gamma_X) = 0$, then
$M$ and $N$ have the same lens data,
and hence $N$ is a 2-disk.
Thus we shall assume that $L > 0$ in this section.

Pick a point $p_0 \in \partial N$,
and let $\beta_1 : [0, 1] \rightarrow \partial N$
be a constant speed closed curve of multiplicity $1$,
starting and ending at $p_0$.
There are two such curves corresponding to different orientations
but either one is fine.

Fix an orientation of $\partial N$ and let $Y_0(x)$
be the unit vector tangent to $\partial N$ at $x \in \partial N$
such that $Y_0(x)$ and $\partial N$ have the same orientation.
Define $\beta_x : [0, 1] \rightarrow N$ as
\begin{align*}
    \beta_x(t) = \gamma_{Y_0(x)} (L t),
\end{align*}
where $\gamma_{Y_0(x)}$ is the closed unit speed geodesic tangent to $Y_0(x)$ of length $L$.
Write $\beta_2 = \beta_{p_0}$.

For any loop $\beta$ in $N$ based at $p \in N$,
we will denote by $[\beta]_p$
the based homotopy class of $\beta$.
Also, let $h : \pi_1(N, p) \rightarrow H_1(N, \integer)$
be the abelianization map
which sends based homotopy classes to corresponding homology classes.
We will write $[\beta] := h([\beta]_p)$.
\begin{prop}
  $[\beta_1]_{p_0} = [\beta_2]_{p_0}^{-2}$.
  \label{prop:square}
\end{prop}
\begin{proof}
  We shall prove the equivalent statement
  \begin{align}
    [\beta_2]_{p_0} = [\beta_2]_{p_0}^{-1} [\beta_1]_{p_0}^{-1}.
    \label{eq_beta}
  \end{align}
  Let $Y : [0, 1]_{p_0} \rightarrow \partial_+ \Omega N$
  be a smooth curve from $Y_0(x)$ to $-Y_0(x)$ such that
  $\gamma_{Y_t}(\tau(Y_t)) = \beta_1(t)$.

  Define $H : [0, 1] \times [0, 1] \rightarrow M$ as
  \begin{align*}
    H_s(t) =
    \begin{cases}
      \gamma_{Y_s}(2 \tau(Y_s) t) & \text{if $0 \le t \le \frac{1}{2}$,}\\
      \beta_1( (2 - 2t)s ) & \text{if $\frac{1}{2} \le t \le 1$.}
    \end{cases}
  \end{align*}
  Then $[H_0]_{p_0} = [\beta_2]_{p_0}$ and $[H_1]_{p_0} = [\beta_2]_{p_0}^{-1} [\beta_1]_{p_0}^{-1}$,
  which implies \eqref{eq_beta}.
\end{proof}
Notice that $\beta_x$ are all in the same homology class.
Denote by $g_C$ the homology class of $\beta_x$.

Assume that $g_C \neq 0$.
Since $N$ is a surface with boundary,
it deformation retracts to a graph.
(The deformation is quite simple.
Take any cell structure on $N$.
Remove a 1-cell on the boundary and a 2-cell by deformation retraction if they intersect.
Repeat this process until all 2-cells are removed.)
Hence $H_1(N, \integer) = \integer^n$ for some $n \in \nonneg$.
So $g_C = m g_0$ for some $m > 0$ and $g_0$ prime.
Then the multiplicity of $\beta_x$ is at most $m$
since it must divide $m$.
Let $m_0$ be the maximal multiplicity of $\beta_x$.
\begin{prop}
    If $g_C \neq 0$, then $H(N, \integer)$ is generated by $g_0$.
    \label{prop:generate}
\end{prop}

\begin{proof}
    For any $g \in \pi_1(N, p_0)$,
    let $\gamma_g : [0, 1] \rightarrow N$ be the length minimizing representative of $g$ that is of constant speed $T_g$.
    Since $A := \gamma_g^{-1} (N \setminus \partial N)$ is open,
    $A = \bigcup \mathcal{A}$ where $\mathcal{A}$ is a family of disjoint open intervals.
    For any $(a, b) \in \mathcal{A}$,
    since $\gamma_g$ is length minimizing,
    $\gamma_g|_{[a, b]}$ has to be a geodesic segment.
    If $a \neq 0$,
    then $\gamma_g'(a)$ has to be tangent to $\partial N$,
    or $\gamma_g$ will have a corner at $\gamma_g(a)$,
    contradicting the assumption that $\gamma_g$ is length minimizing.
    According to the scattering data,
    $\gamma_g'(b)$ is also tangent to $\partial N$ if $\gamma_g'(a)$ is tangent to $\partial N$.
    Hence $\gamma_g|_{[a, b]}$ is a closed geodesic tangent to $\partial N$ when $a \neq 0$.
    Similarly, $\gamma_g|_{[a, b]}$ is closed geodesic tangent to $\partial N$ when $b \neq 1$.
    Suppose that $a = 0$ and $b = 1$.
    If $\gamma_g'(0)$ is not tangent to $\partial N$,
    then $\gamma_g'(0) / |\gamma_g'(0)| = \varphi(X) \in \partial_+ \Omega N$ for some $X \in \partial_+ \Omega M$,
    and we have $\gamma_g'(1) / |\gamma_g'(1)| = \alpha_M(X)$.
    Since $M$ is a simple manifold and $X \in \partial_+ \Omega M$,
    $\gamma_X$ is a length minimizing geodesic,
    and thus
    $X$ and $\alpha_M(X)$ have different base points.
    It follows that $\gamma_g'(0)$ and $\gamma_g'(1)$ also have different base points,
    contradicting our assumption that $\gamma_g$ is a loop.
    Therefore, in any case,
    $\gamma_g|_{[a, b]}$ is a closed geodesic tangent to $\partial N$,
    and thus of length at least $\frac{L}{m_0}$.
    It follows that $|\mathcal{A}| \le m_0 T_g / L < \infty$.
    So we can write $\mathcal{A} = \{(a_1, b_1), (a_2, b_2), \dots, (a_{n_g}, b_{n_g})\}$
    where $n_g = |\mathcal{A}|$ and
    $0 \le a_1 < b_1 \le a_2 < b_2 \le \dots \le a_{n_g} < b_{n_g} \le 1$.

    Since $\gamma_g|_{[a_i, b_i]}$
    is a closed geodesic,
    $[\gamma_g|_{[a_i, b_i]}] = g_0^{k_i}$ for some $k_i > 0$.
    Deleting all those closed geodesics from $\gamma_g$,
    we obtain a curve running around $\partial N$ $l$ times
    for some $l \in \integer$.
    Its homology class will be $[\beta_1]^{\pm l} = [\beta_2]^{\pm 2l} = g_0^{\pm 2ml}$.
    Therefore,
    $h(g) = g_0^{\pm 2m_0l + \sum_{i = 1}^{n_g} k_i}$.
    Since $h$ is surjective, $H_1(N, \integer)$ is generated by $g_0$.
\end{proof}

\begin{prop}
  $N$ is not a M\"obius strip.
  \label{prop:not_mo}
\end{prop}
\begin{proof}
  Let $\pi : N_1 \rightarrow N$ be a double over of $N$.
  Then $N_1$ is an annulus with two boundary components $S_1$ and $S_2$.
  There are $p \in S_1$ and $q \in S_2$ such that $d(p, q) = d(S_1, S_2)$.
  Let $\gamma$ be the shortest curve from $p$ to $q$,
  then $\gamma$ is perpendicular to $S_1$ and $S_2$ at tis end points.
  Let $\nu$ be the unit normal vector at $p$.
  Since $\gamma$ is the shortest curve from $p$ to $q$,
  its beginning part must coincide with $\gamma_{\nu}$.
  If the end point $\gamma_{\nu}(\tau_{N_1}(\nu))$
  is on $S_1$,
  we can shorten $\gamma$ by deleting $\gamma_\nu$.
  If the end point $\gamma_{\nu}(\tau_{N_1}(\nu))$
  is on $S_2$,
  then $\gamma$ can not be any longer.
  Thus $\ell(\gamma) = \tau_{N_1}(\nu)$
  
  Notice that, for any $X \in \partial_+ \Omega N_1$,
  $\pi \circ \gamma_X = \gamma_{\pi_*(X)}$.
  Let $Y_t$ be a smooth curve in $\Omega_p N_1$ such that
  $Y_0 = \nu$, that $Y_t \in \partial_+ \Omega N_1$ for $t \in [0, 1)$ and
  that $Y_1$ is tangent to $\partial S_1$.
  Notice that the end point of $\pi \circ \gamma_{Y_t} = \gamma_{\pi_*(Y_t)}$ moves continuously (since $N$ has the same scattering data as the simple surface $M$),
  and hence the end point of $\gamma_{Y_t}$ moves continuously for $t \in [0, 1)$.
  Therefore, $\gamma_{Y_t}$ connects $p$ and $S_2$ for any $t \in [0, 1)$.
  However,
  we have $\ell(\gamma) = \tau_{N_1}(\nu) = \tau_{N}(\pi_*(\nu)) > L$,
  and $\lim_{t \rightarrow 1} \ell(\gamma_{Y_t}) = \lim_{t \rightarrow 1} \tau_{N}(\pi_*(Y_t)) = L$,
  which contradicts our assumption that $\gamma$ is a shortest curve connecting $S_1$ and $S_2$.
\end{proof}

\begin{proof}[Proof of Proposition \ref{prop:disk}]
  If $g_C \neq 0$,
  then $H_1(N, \integer)$ is generated by $g_0$ by Proposition \ref{prop:generate}.
  Hence $H(N, \integer) = \integer$,
  which implies that $N$ is a M\"obius strip,
  which contradicts Proposition \ref{prop:not_mo}.

  Therefore, $g_C = 0$.
  It follows that $\beta_2$ is contractible,
  and hence $\beta_1$ is contractible by Proposition \ref{prop:square}.
  Since every contractible simple closed curve on a surface bounds a disk \cite[Theorem 1.7]{epstein1966curves},
  $N$ is a disk.
\end{proof}

\section{Knot theory}
In this chapter,
$N$ will denote a Riemannian surface, with or without boundary, orientable or not.
We assume that there is a Riemannian metric on $N$
just for convenience and all the results can be stated with only a smooth structure.
\label{sec:knot}
\subsection{Projectivized unit tangent vector fields}
  \begin{defn}
    The \emph{unit tangent vector field}  of
    a smoothly immersed curve $\gamma$ on any 
    Riemannian surface $N^2$ (possibly with boundary) is a smoothly immersed curve
    $\tilde{\gamma}$ in $\Omega N$ defined as
    \begin{align}
      \tilde{\gamma}(t) = \left(\gamma(t), \frac{\gamma'(t)}{|\gamma'(t)|}\right).
      \label{eq:unit_tangent}
    \end{align}
  \end{defn}
  \begin{defn}
    Let $P : \Omega N \rightarrow P \Omega N$ be the quotient map
    on the unit tangent bundle which identifies the opposite vectors based
    at the same point.
    For any smoothly immersed curve
    $\gamma$ in $N^2$,
    $P \circ \tilde{\gamma}$ is called
    the \emph{projectivized unit tangent vector field} (or the \emph{tangent line field}) of $\gamma$.
  \end{defn}

  \begin{rem}
    Chmutov--Goryunov--Murakami \cite{MR1776110} showed that every knot type in $\Omega \mathbb{R}^2$ is realized by the unit tangent vector field along an immersed closed plane curve. However, Theorem \ref{thm:knot} shows that it is no longer possible to realize the trivial knot after the projectivization. Figure \ref{fig:8} is an interesting example showing that the unit tangent vector field of the figure eight curve is an unknot while the projectivized unit tangent vector field of the
figure eight curve is knotted.
  \end{rem}
  \begin{figure}[h]
    \centering
    \begin{subfigure}[t]{0.28\textwidth}
      \centering
      \includegraphics[width=\textwidth]{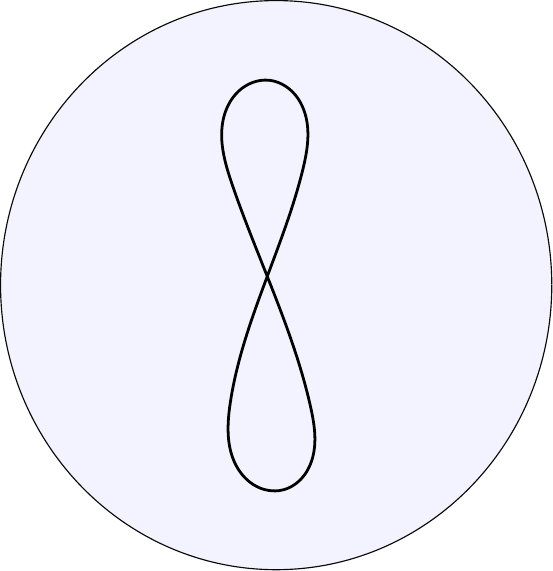}
      \caption{$\gamma : \sphere \rightarrow \real^2$}
    \end{subfigure}%
    \quad 
    \begin{subfigure}[t]{0.28\textwidth}
      \centering
      \includegraphics[width=\textwidth]{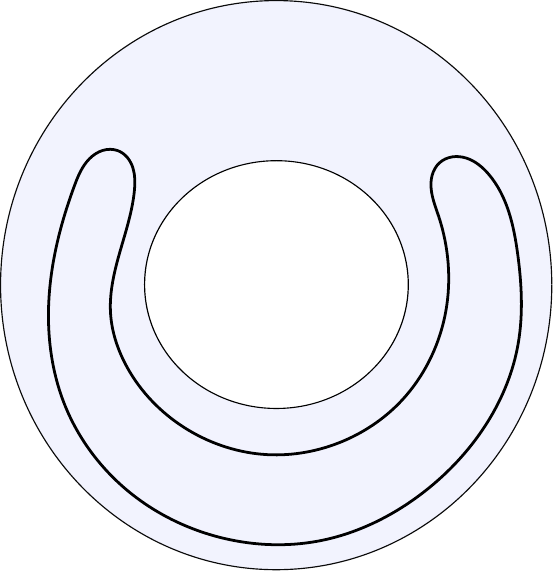}
      \caption{$\tilde\gamma : \sphere \rightarrow \Omega\real^2$}
    \end{subfigure}%
    \quad
    \begin{subfigure}[t]{0.28\textwidth}
      \centering
      \includegraphics[width=\textwidth]{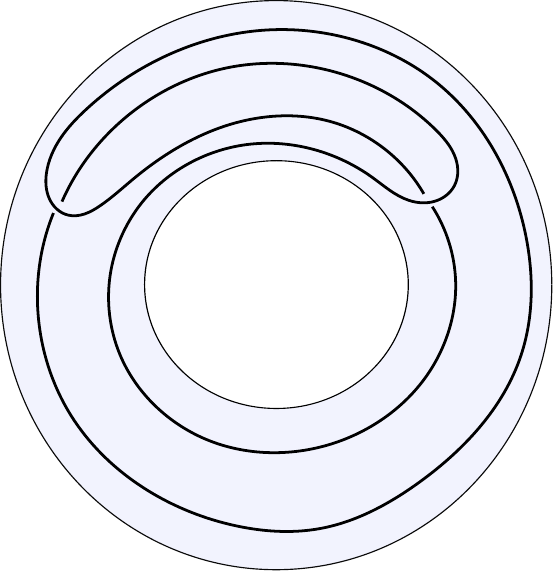}
      \caption{$P \circ \tilde{\gamma} : \sphere \rightarrow P \Omega \real^2$}
    \end{subfigure}
    \caption{The unit tangent vector field of the figure eight curve is unknotted while the projectivized unit tangent vector field of the figure eight curve is knotted. Here the solid tori ($\Omega \real^2$ and $P \Omega \real^2$) are projected to annuli for illustration.}
    \label{fig:8}
  \end{figure}

  \begin{prop}
    For any smoothly \textbf{embedded} closed curve $\gamma : \sphere \rightarrow N$
    in a 2-dimensional manifold $N$,
    $\tilde{\gamma}$ is not contractible in $\Omega N$ and hence
    $P \circ \tilde{\gamma}$ is not contractible in $P \Omega N$.
  \label{prop:embedded}
  \end{prop}
  \begin{proof}
    If $\gamma$ is contractible,
    then $\gamma$ bounds an embedded disk $N_1$ in $N$ \cite[Theorem 1.7]{epstein1966curves}.

    Let $x = \tilde{\gamma}(0)$,
    $p = \gamma(0)$
    and $F = \pi^{-1}(p)$.
    Denote by $[\tilde\gamma]_x$
    the based homotopy class of $\tilde\gamma$.
    \begin{figure}[h]
      \center
      \begin{subfigure}[t]{0.35\textwidth}
	\centering
	\includegraphics[width=\textwidth]{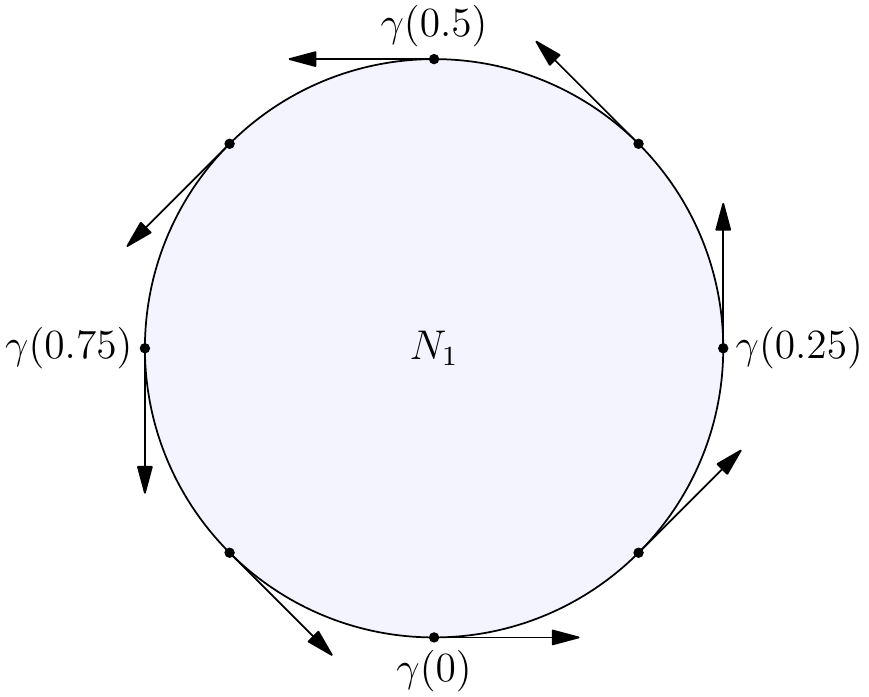}
	\caption{$\tilde\gamma$}
	\label{}
      \end{subfigure}%
      \quad
      \begin{subfigure}[t]{0.25\textwidth}
	\centering
	\includegraphics[width=\textwidth]{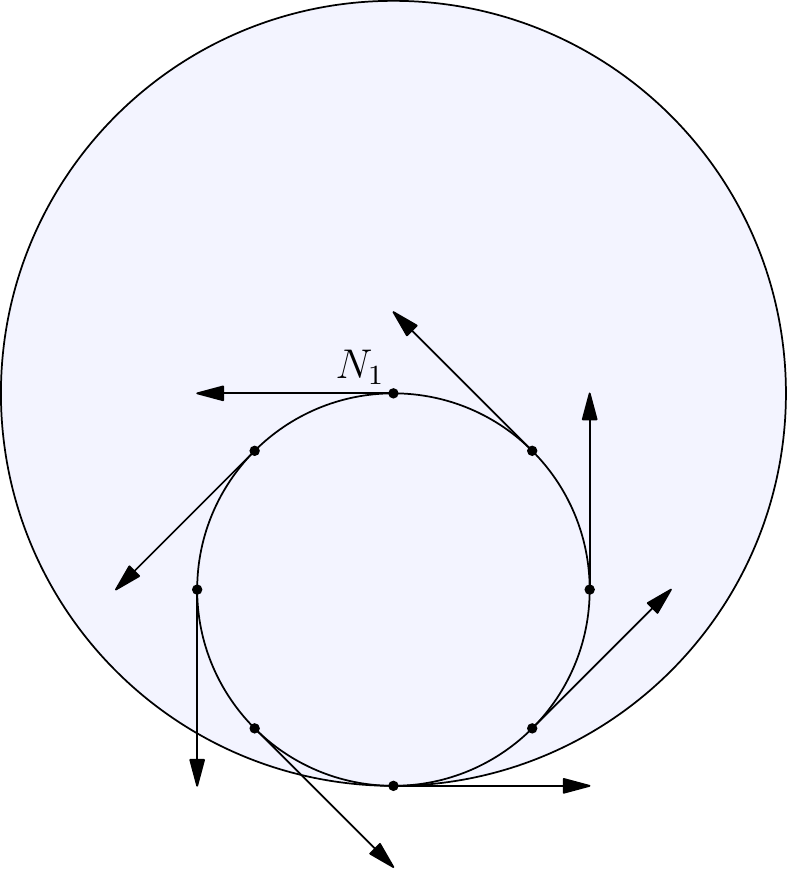}
	\caption{Moving base points towards $p$}
	\label{}
      \end{subfigure}
      \quad
      \begin{subfigure}[t]{0.25\textwidth}
	\centering
	\includegraphics[width=\textwidth]{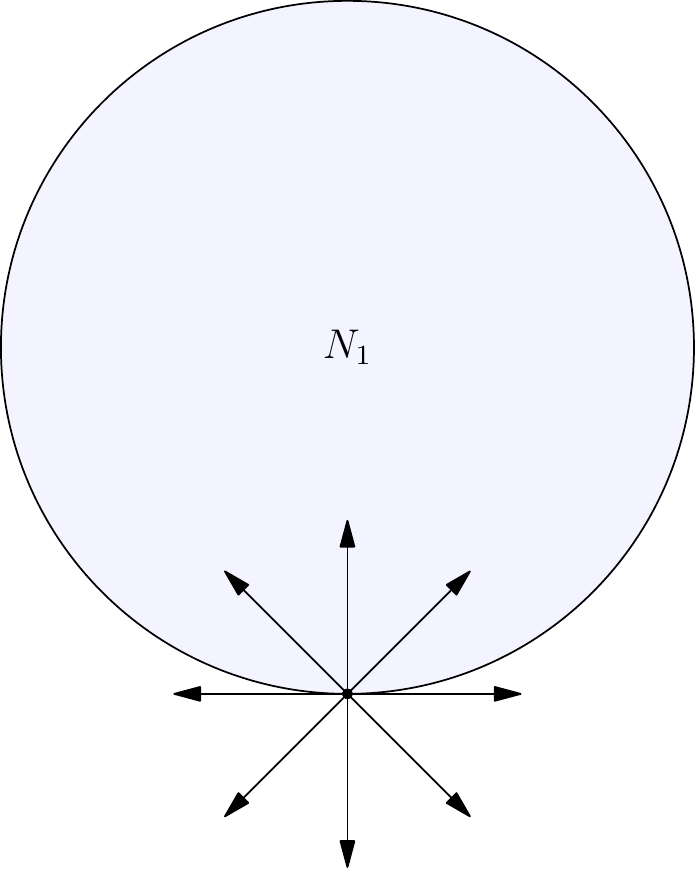}
	\caption{A generator of $\pi_1(F, x)$}
	\label{}
      \end{subfigure}
      \caption{$\tilde\gamma$ is homotopic to a generator of $\pi_1(F, x)$}
      \label{fig_rotate}
    \end{figure}
    As in Figure \ref{fig_rotate}, $\tilde\gamma$ corresponds to a vector moving along $\gamma$ for a complete circle and being tangent to $\gamma$ all the time,
    which is homotopic to a generator
    of $\pi_1(F, x)$.
    Denote the generator of $\pi_1(F, x)$ by $g_1$.

    Since 
    \begin{align*}
      \begin{CD}
	F @>i>> \Omega N @>\pi>> N
      \end{CD}
    \end{align*}
    is a fibration,
    we have an exact sequence of homotopy groups
    \begin{align*}
      \begin{CD}
	\pi_2(N, p) @>>> \pi_1(F, x) @>i_*>> \pi_1(\Omega N, x) @>\pi_*>> \pi_1(N, p).
      \end{CD}
    \end{align*}
    Here $\pi_1(F, x) = \integer$
    since $F$ is a circle.

    If $N = \mathbb{S}^2$,
    then $\pi_1(N, p) = 0$
    and $\pi_1(\Omega N, p) = \pi_1(\real P^3, *) = \integer / 2 \integer$.
    Hence $i_*(\pi_1(F, x)) = \integer / 2 \integer$.
    In particular, $i_*(g_1) \neq 0$.
    A similar argument shows that $i_*(e) \neq 0$ when $N = \real P^2$.

    If $N \neq \mathbb{S}^2$ and $N \neq \real P^2$,
    then $\pi_2(N, p) = 0$.
    Hence $i_*$ is injective.
    In particular, $i_*(g_1) \neq 0$.

    This completes the proof of Proposition \ref{prop:embedded}
  \end{proof}

  \subsection{Knot invariants}
  We shall define a family of knot invariants for contractible knots in the projectivized unit tangent bundle $P \Omega N$ and use these invariants to prove Theorem \ref{thm:knot}.

  Let $\beta : \sphere \rightarrow P \Omega N$ be a contractible smooth knot in the projectivized unit tangent bundle $P \Omega N$,
  whose projection to the surface $N^2$ is a smoothly immersed curve $\gamma :
  \sphere \rightarrow N^2$ without self-tangencies.
  \begin{defn}
    $\beta$ has a \emph{crossing} at $(l, l') \in \sphere \times \sphere$
    if $l \neq l'$ and $\gamma(l) = \gamma(l')$.
    Note that a triple crossing will be treated as three independent
    crossings according to this definition.
  \end{defn}
  Since $\beta : \sphere \rightarrow P \Omega N$ is contractible, we can lift $\beta$ to $\hat\beta : \sphere \rightarrow \Omega N$, a knot embedded in the unit tangent bundle. ($\hat\beta(t)$ is a unit vector at $\gamma(t)$ but not necessarily tangent to $\gamma$.)
  \begin{defn}
    A crossing of $\beta$ at $(l, l')$ is \emph{positive}
    if the two pairs of vectors $(\hat\beta(l), \hat\beta(l'))$
    and $(\gamma'(l), \gamma'(l'))$ are of the same orientation. (See Figure \ref{fig:sign}.)
    A crossing will be called \emph{negative} if it is not positive.
  \end{defn}
  \begin{figure}[h]
    \centering
    \begin{subfigure}[b]{0.4\textwidth}
      \centering
      \includegraphics[width=0.7\textwidth]{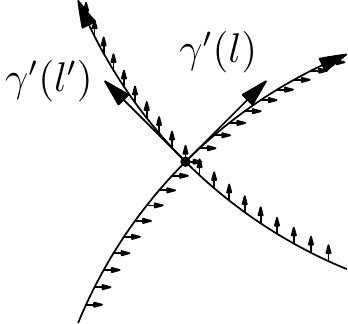}
      \caption{A positive crossing}
      \label{fig:pint}
    \end{subfigure}%
    \quad
    \begin{subfigure}[b]{0.4\textwidth}
      \centering
      \includegraphics[width=0.7\textwidth]{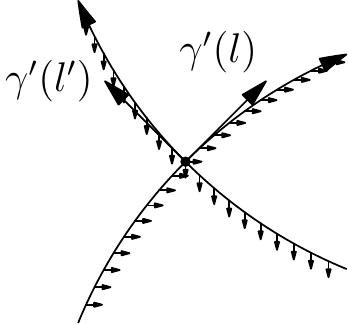}
      \caption{A negative crossing}
      \label{fig:nint}
    \end{subfigure}
    \caption{Each little arrow means a point on $\hat\beta$}.
    \label{fig:sign}
  \end{figure}
  \begin{lem}
    Suppose that $\gamma : \sphere \rightarrow N$ is a smoothly immersed closed curve on a surface $N$ without self-tangencies,
    then all crossings of $P \circ \tilde{\gamma}(t)$ are positive.
    \label{lem:positive}
  \end{lem}
  \begin{proof}
    Suppose that $P \circ \tilde{\gamma}$ has a crossing at $(l, l')$.
    Write $\beta = P \circ \tilde{\gamma}$.
    Then $(\hat\beta(l), \hat\beta(l')) = (\gamma'(l), \gamma'(l'))$,
    and hence they have the same orientation.
    Therefore the crossing at $(l, l')$ is positive.
  \end{proof}

  Let $X$ be any topological space.
  For any two curves $\alpha_1 : [0, 1] \rightarrow X$ and $\alpha_2 : [0, 1] \rightarrow X$
  such that $\alpha_1(1) = \alpha_2(0)$,
  denote by $\alpha_1 * \alpha_2 : [0, 1] \rightarrow X$ the curve obtained by gluing $\alpha_2$ to $\alpha_1$.
  Also, define $R(\alpha_1)$ as $R(\alpha_1)(t) := \alpha_1(1-t)$.
  If $\alpha_1$ and $\alpha_2$ are loops based at $p$, we have $[\alpha_1]_p[\alpha_2]_p = [\alpha_1 * \alpha_2]_p$
  and $[R(\alpha_1)]_p = [\alpha_1]_p^{-1}$.

  \begin{defn}
    Two closed curves $\alpha_1 : \sphere \rightarrow X$ and $\alpha_2 : \sphere \rightarrow X$ in any topological space $X$ are said to be in the same \emph{unoriented free homotopy class}
    if $\gamma_1$ is homotopic to either $\gamma_2$ or $R(\gamma_2)$.
    \label{defn:ufhc}
  \end{defn}
  When $\beta$ has a crossing at $(l, l')$, $\hat\beta(l)$ and $\hat\beta(l')$ are two unit vectors with
  the same base point $x = \pi(\beta(l))$
  and they are neither opposite to each other nor the same
  (since $\beta$ is an embedding).
  Hence there is a unique shortest curve $\hat\beta_{(l, l')}$ in $\pi^{-1}(x)$ connecting $\hat\beta(l)$ and $\hat\beta(l')$.
  Separate $\hat\beta$ into two arcs by cutting at $\hat\beta(l)$ and $\hat\beta(l')$,
  obtaining two arcs
  $\hat\beta_1 : [0, 1] \rightarrow P \Omega N$
  and $\hat\beta_2 : [0, 1] \rightarrow P \Omega N$
  going from $\hat\beta(l)$ to $\hat\beta(l')$.

  Now,
  let $\beta_1' = (P \circ \hat\beta_1) * R(P \circ \hat\beta_{(l, l')})$,
  and $\beta_2' = (P \circ \hat\beta_{(l, l')}) * R(P \circ \hat\beta_2)$.
  Notice that $\beta_1' * R(\beta_2')$ is homotopic to $\beta$,
  and hence $[\beta_1']_p[R(\beta_2')]_p = [\beta_1' * R(\beta_2')]_p = [\beta]_p = e$.
  Hence $[\beta_1']_p = [R(\beta_2')]_p^{-1} = [\beta_2']_p$.
  In other words,
  $\beta_1'$ is homotopic to $\beta_2'$,
  and hence $\beta_1'$ and $\beta_2'$ are in the same unoriented free homotopy class of $P \Omega N$.

  \begin{defn}
  The unoriented free homotopy class $g_{(l, l')}$ of $\beta_1'$
  is called the \emph{type} of the crossing of $\beta$ at $(l, l')$. 
  \end{defn}
  \begin{figure}[h]
    \center
    \includegraphics[width=0.60 \textwidth]{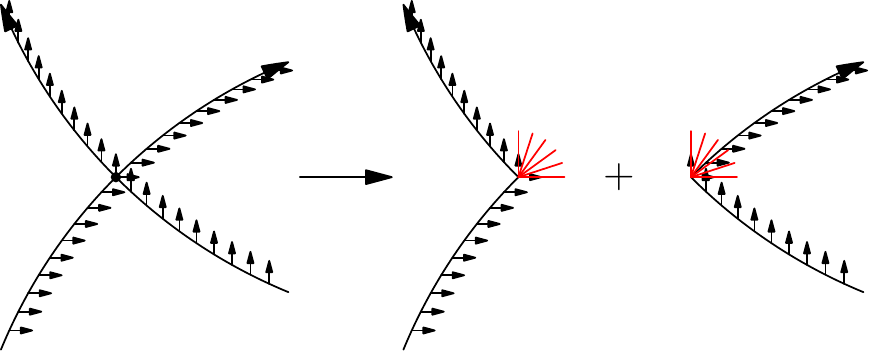}
    \caption{Smoothing a crossing. Here each little arrow means a point on
      $\hat\beta$ and each little bar means a point on $\hat\beta_{(l, l')}$.}
    \label{fig:smoothing}
  \end{figure}

  \begin{defn}
    For each \textbf{nontrivial} unoriented free homotopy class $g$ of closed curves in the projectivized unit tangent bundle $P \Omega N$,
    define
    \begin{align*}
      W_g(\beta) &= \#\{\text{positive crossings of $\beta$ of type $g$}\}\\
      &\quad -\#\{\text{negative crossings of $\beta$ of type $g$}\}
    \end{align*}
  \end{defn}

  \subsection{$W_g$ is a knot invariant.}
  \begin{thm}
    For each non-trivial free homotopy class $g$,
    $W_g$ can be extended to all the contractible knots embedded in $P \Omega N$ as a knot invariant.
    \label{thm:inv}
  \end{thm}
  We will show that $W_g$ is a knot invariant by verifying that $W_g$ is unchanged under Reidemeister moves. A knot will gain or lose a crossing of trivial type after going through a Reidemeister move of type I. It will gain or lose a pair of crossing of the same type but opposite signs after going through a Reidemeister move of type II. Reidemeister moves of type III will not affect crossing. The proof is rather lengthy because of some technical difficulties.
 % We could have a much shorter proof if we assume that $N$ is homeomorphic to $\real^2$, which is the case in this paper. However, we will work with any complete surface $N$ so the result can be applied to later papers.

    We will assume that $N$ is compact,
    and the general case follows automatically since any manifold is $\sigma$-compact.
    
    \begin{defn}
      According to \cite[Theorem 5]{alexander1987},
      there is $r > 0$ such that there is a unique minimal geodesic segment joining $p, q \in N$ if $d(p, q) < r$.
      The biggest such $r$ will be called the \emph{injectivity radius} of $N$ and we will denote it by $inj(N)$.
    \end{defn}
    
    For any two points $p, q \in P \Omega N$,
    let $d_h(p, q)$ be the distance between $\pi(p)$ and $\pi(q)$ on $N$.
    (So $d_h$ is a pseudo metric on $P \Omega N$.)
    Notice that $p$ is a projectivized unit tangent vector at $\pi(p)$.
    When $d_h(p, q) < inj(N)$,
    there is a unique shortest geodesic $\gamma : [0, 1] \rightarrow N$ in $N$ connecting $\pi(p)$ and $\pi(q)$.
    Let $X : [0, 1] \rightarrow P \Omega N$
    be the parallel projectivized vector field along $\gamma$
    such that $X(0) = p$.
    Similarly,
    let $Y : [0, 1] \rightarrow P \Omega N$
    be the parallel projectivized vector field along $\gamma$
    such that $Y(1) = q$.
    Notice that the angle between $X$ and $Y$ is constant,
    which is smaller or equal to $\frac{\pi}{2}$.
    Call this angle $d_v(p, q)$.
    Next, put $d^0(p, q) = \max(d_h(p, q), d_v(p, q))$.
    Note that $d_h$, $d_v$ and $d^0(p, q)$ are all non-negative and symmetric,
    but they are not metrics.
    \begin{defn}
      For any $p, q \in P \Omega N$ such that $d_h(p, q) < inj(N)$ and that
      $d_v(p, q) < \frac{\pi}{2}$,
      let $\gamma^0_{p, q} : [0, 1] \rightarrow \Omega N$ be the curve that satisfies the following conditions.
      \begin{enumerate}
	\item $\gamma^0_{p, q}(0) = p$ and $\gamma^0_{p, q}(1) = q$.
	\item $\pi \circ \gamma$ is the minimal geodesic connecting $\pi(p)$ and $\pi(q)$.
	\item
	\begin{align}
	  \left|\frac{D}{dt} \gamma^0_{p, q}\right| = d_v(p, q).
	  \label{eq:dv}
	\end{align}
      \end{enumerate}
      $\gamma^0_{p, q}$ will be called the \emph{minimal linear curve} connecting $p$ and $q$.
      A curve will be called \emph{linear} if it coincides with $\gamma^0_{p, q}$ for any pair of points $p, q$ on the curve that are close enough.
      A curve will be called \emph{piecewise linear} if it consists of finitely many linear curves.
    \end{defn}

  \begin{proof}[Proof of Theorem \ref{thm:inv}]

  For any $\varepsilon < \min(inj(N), \frac{\pi}{2})$ and $n \ge 4$,
  we will define a class of closed piecewise linear knots in $P \Omega N$ called $\mathcal{K}(n, \varepsilon)$.
  A closed knot $\beta : \real / \integer \rightarrow P \Omega N$ is in $\mathcal{K}(n, \varepsilon)$ if and only if the following condition holds:
  \begin{enumerate}
    \item $\beta$ is contractible.
    \item $d^0(\beta(\frac{k}{n}), \beta(\frac{k+1}{n})) < \varepsilon$ for $k = 0, 1, \dots, n-1$.
    \item $\beta(\frac{k+t}{n}) = \gamma^0_{\beta(\frac{k}{n}), \beta(\frac{k+1}{n})}(t)$ for $t \in [0, 1]$ and $k = 0, 1, \dots, n-1$.
  \end{enumerate}

  In other words,
  the ``distance'' ($d_h$ and $d_v$) between any two adjacent vertices $p$ and $q$ is at most $\varepsilon$ and the edge between them is $\gamma^0_{p, q}$.
  $\mathcal{K}(n, \varepsilon)$ is
  an open subset of $(P \Omega N)^n$,
  and thus of dimension $3n$.

  Let $\beta \in \mathcal{K}(n, \varepsilon)$ be a piecewise smooth \textbf{contractible} knot
  with vertices $\{x_k = \beta(\frac{k}{n})\}$ and edges $\{e_k = \gamma^0_{x_k, x_{k+1}}\}$.
  Its projection $\pi \circ \beta$ is said to have a \emph{singularity}
  at the vertex $\pi(x_i)$
  if $\pi(x_i)$ is on $\pi \circ e_j$ for some $j \notin \{i, i-1\}$.

  Let $\mathcal{K}_k(n, \varepsilon)$ be the set of knots in $\mathcal{K}(n, \varepsilon)$ whose projections on $N$ have at most $k$ singularities.
  Also,
  let $\mathcal{K}'_k(n, \varepsilon) = \mathcal{K}_k(n, \varepsilon) - \mathcal{K}_{k-1}(n, \varepsilon)$,
  knots with exactly $k$ singularities.
  Then $\mathcal{K}_0(n, \varepsilon)$ is a open submanifold of
  $\mathcal{K}(n, \varepsilon)$,
  and $\mathcal{K}'_1(n, \varepsilon)$ is a submanifold of
  $\mathcal{K}(n, \varepsilon)$ of co-dimension $1$.

    Notice that the $W_g(\beta)$ can be defined for $\beta \in \mathcal{K}_0(n, \varepsilon)$ as before without any modifications.
    Consider a continuous family of knots $\beta_t \in \mathcal{K}_0(n, \varepsilon)$.
    As $t$ varies,
    crossings of $\beta_t$ also moves continuously
    with their types unchanged.
    Therefore, $W_g$ is constant on each component of $\mathcal{K}_0(n, \varepsilon)$.

    Next, we extend $W_g$ to $\mathcal{K}_1(n, \varepsilon)$.
    The old definition can not be adapted directly since there might singularities.
    Pick any $\beta_0, \beta_1 \in \mathcal{K}_0(n, \varepsilon)$
    such that $\beta_0, \beta_1$ are in the same component of $\mathcal{K}_1(n, \varepsilon)$.
    We aim to show that $W_g(\beta_0) = W_g(\beta_1)$,
    and then we can extend $W_g$ to $\mathcal{K}_1(n, \varepsilon)$
    by making it constant on each component.
    Note that $W_g$ will remain the same on $\mathcal{K}_0(n, \varepsilon)$.

    Pick a smooth path $H : [0, 1] \rightarrow \mathcal{K}_1(n, \varepsilon)$ from $\beta_0$ to $\beta_1$.
    Perturbing $H$ if necessary,
    we may assume that $H$ intersects $\mathcal{K}'_1(n, \varepsilon)$ transversely a finite number of times.
    Let $x_0(t)$, $x_1(t)$, \dots, $x_n(t) = x_0(t)$ be the $n$ vertices of $H(t)$.
    
    As long as $H(t)$ stays in $\mathcal{K}_0(n, \varepsilon)$,
    each crossing will just be moving without changing its type.
    When $H(t)$ passes through $\mathcal{K}'_1(n, \varepsilon)$,
    there are three possibilities corresponding to three types of singularities for knots in $\mathcal{K}'_1(n, \varepsilon)$ listed below.
    Suppose $H(c) \in \mathcal{K}'_1(n, \varepsilon)$
    and $H(t) \notin \mathcal{K}'_1(n, \varepsilon)$
    for $t \in (c - \delta, c) \bigcup (c, c + \delta)$.
    Then $\pi \circ H(c)$ has a singularity at $\pi(x_i(c))$ which is on $\pi \circ e_j$
    where $x_i(c)$ is the $i$-th vertex of $H(c)$ and $e_j(c)$ is $j$-th edge of $H(c)$ (connecting $x_j(c)$ and $x_{j+1}(c)$).
    \begin{enumerate}
      \item If $\pi \circ e_i(c)$ or $\pi \circ e_{i-1}(c)$ is tangent to $\pi \circ e_j(c)$,
	then the singularity is called a \emph{cusp}.
	This happens when $i = j-1$ or $i = j+2$.
	In this case, $H(c + \delta)$ has one more or one less crossing than
	$H(c - \delta)$ has.
	We will show that the crossing involved is of the trivial type,
	(i.e., $g = 0$,)
	and hence $W_g(H(c-\delta)) = W_g(H(c+\delta))$ for any non-trivial unoriented homotopy class $g$ of closed curves immersed in $P \Omega N$.
\begin{figure}[h]
    \centering
    \begin{subfigure}[b]{0.15\textwidth}
      \centering
      \includegraphics[width=\textwidth]{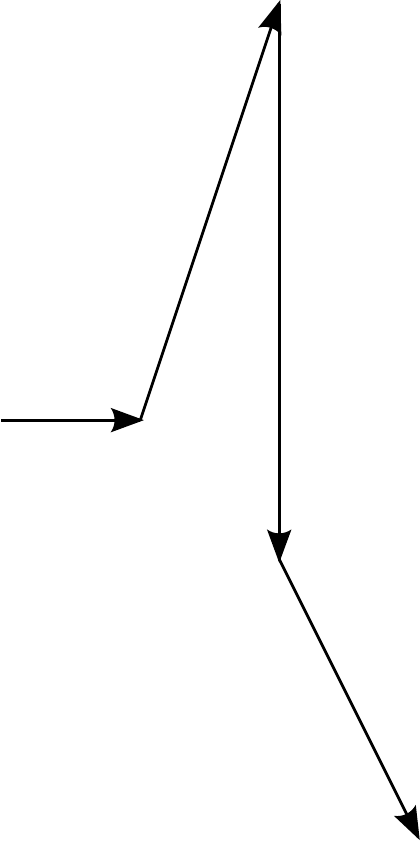}
      \caption{$t = c - \delta$}
      \label{fig:b1_l}
    \end{subfigure}%
    \quad
    \quad
    \quad
    \begin{subfigure}[b]{0.15\textwidth}
      \centering
      \includegraphics[width=\textwidth]{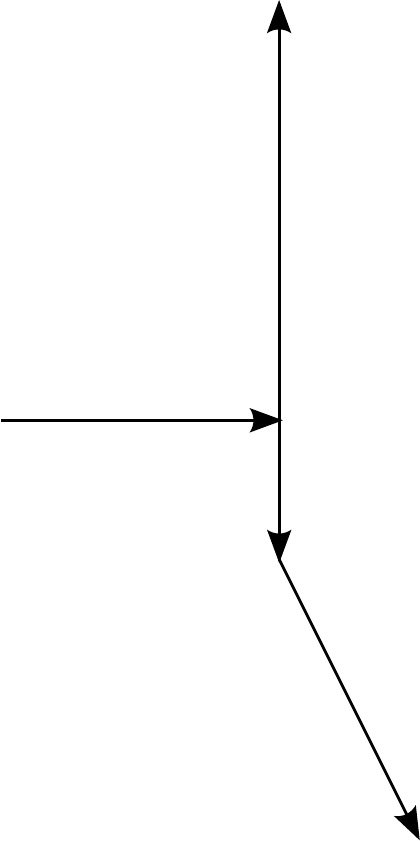}
      \caption{$t = c$}
      \label{fig:b1_0}
    \end{subfigure}
    \quad
    \quad
    \quad
    \begin{subfigure}[b]{0.15\textwidth}
      \centering
      \includegraphics[width=\textwidth]{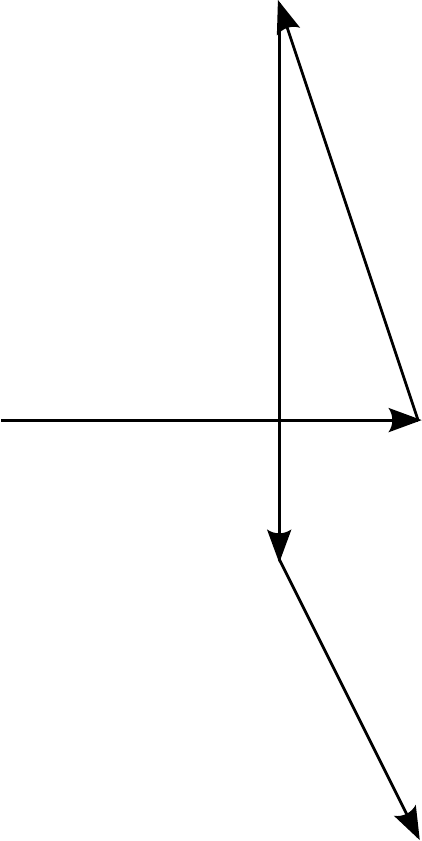}
      \caption{$t = c + \delta$}
      \label{fig:b1_r}
    \end{subfigure}
    \caption{\ref{fig:b1_r} has one more crossing of the trivial type compared to \ref{fig:b1_l}.
    This corresponds to a Reidemeister move of type I.}
    \label{fig:r1}
\end{figure}

	Without loss of of generality,
	assume that $i = j-1$ and 
	that $H(t)$ has one more crossing at $(l(t), l'(t))$ than
	$H(t')$ has when $c - \delta \le t' < c < t \le c + \delta$.
	(See Figure \ref{fig:r1}.)
	For any $t \in (c, c+\delta]$,
	swapping $l(t)$ and $l'(t)$ if necessary,
	we may
	assume that $l(t) \in (\frac{i-1}{n}, \frac{i}{n})$
	and $l'(t) \in (\frac{i+1}{n}, \frac{i+2}{n})$.
	Lift $H : [0,1] \rightarrow \mathcal{K}(n, \varepsilon)$ to $\hat{H} : [0, 1] \rightarrow (\sphere \rightarrow \Omega N)$.
	Since $H(t)$
	is an embedding,
	$H(t)(l(t)) \neq H(t)(l'(t))$,
	and hence $\hat{H}(t)(l(t))$ and $\hat{H}(t)(l'(t))$
	are not opposite vectors.
	It follows that there is a unique minimal geodesic $\hat{\alpha}(t) : [0, 1] \rightarrow \pi^{-1}(x)$ connecting
	$\hat{H}(t)(l(t))$ and $\hat{H}(t)(l'(t))$.
	Let $\alpha(t) = P \circ \hat\alpha(t)$
	and glue $\alpha(t)$ to $H(t)|_{[l(t), l'(t)]}$,
	obtaining a closed curve $C(t)$.
	Then the type of the crossing of $H(t)$ at $(l(t), l'(t))$
	is the unoriented homotopy class of $C(t)$.
	It remains to show that $C(t)$ is contractible.

	Let $l(c) = \lim_{t\rightarrow c+}l(t)$ and 
	$l'(c) = \lim_{t\rightarrow c+}l'(t)$,
	then
	$C(c)$ can be defined as before,
	which is homotopic to $C(t)$ for
	$t \in (c, c + \delta)$.
	We shall show that $C(c)$ is contractible.

	Reparametrize $C(c)$ as $\bar\beta : \real / \integer \rightarrow P \Omega N$
	such that $\bar\beta(0) = H(c)(l(c))$,  $\bar\beta(\frac{1}{3}) = x_{i+1}(c)$,
	$\bar\beta(\frac{2}{3}) = H(c)(l'(c))$ and $\pi(\bar\beta(\frac{1}{3}(1+s))) = \pi(\bar\beta(\frac{1}{3}(1-s)))$ for any $s \in [0, 1]$.
	To be precise,
	define $\bar\beta$ as
	\begin{align*}
	  \bar\beta(t) =
	  \begin{cases}
	    H(c)(\frac{i+3t}{n}) & \text{if $t \in [0, \frac{1}{3}]$,}\\
	    H(c)(\frac{i+1}{n} + (3t - 1)(l'(c) - \frac{i+1}{n})) & \text{if $t \in [\frac{1}{3}, \frac{2}{3}]$,}\\
	    \alpha(3 - 3t) & \text{if $t \in [\frac{2}{3}, 1]$.}
	  \end{cases}
	\end{align*}
	Consider the homotopy $G : [0, 1] \rightarrow (\sphere \rightarrow P \Omega N)$ defined as
	\begin{align*}
	  G(s)(t) =
	  \begin{cases}
	    \bar\beta(t) & \text{if $0 \le t \le \frac{1}{3}(1-s)$,}\\
	    T(\bar\beta(t), \pi(\bar\beta(\frac{1}{3}(1-s)))) & \text{if $\frac{1}{3}(1-s) \le t \le \frac{1}{3}(1+s)$,}\\
	    \bar\beta(t) & \text{if $\frac{1}{3}(1+s) \le t \le 1$,}
	  \end{cases}
	\end{align*}
	where $T(\bar\beta(t), \pi(\bar\beta(\frac{1}{3}(1-s))))$ is a projectivized unit tangent vector at $\pi(\bar\beta(\frac{1}{3}(1-s)))$ obtained by
	transporting $\bar\beta(t)$ parallelly along $\pi \circ e_j(c)$.
	Notice that $G(1)$ is a closed curve in $\Omega_{\pi(x_i)} N$,
	where $\Omega_{\pi(x_i)} N$ is a circle of length $2 \pi$, (using the Sasakian metric).
	We are going to show that $G(1)$ is contractible by showing that $\ell(G(1)) < 2 \pi$.
	For any piecewise smooth curve $\gamma : [a, b] \rightarrow P \Omega N$,
	define its \emph{vertical length} as
	\begin{align*}
	  \ell_v(\gamma) := \int_a^b \left|\frac{D}{dt}\gamma(t)\right| dt,
	\end{align*}
	where $\frac{D}{dt}$ is the covariant derivative.
	Loosely speaking,
	$\ell_v(\gamma)$ measure the angle that $\gamma(t)$ rotates by as $t$ goes from $a$ to $b$.
	By our construction, $\ell_v(G(s))$ is constant as $s$ goes from $0$ to $1$.
	Notice that $\bar\beta$ has three edges.
	The edge from $\bar\beta(0)$ to $\bar\beta(\frac{1}{3})$ and the edge from $\bar\beta(\frac{1}{3})$ to $\bar\beta(\frac{2}{3})$ both have vertical lengths at most $\varepsilon$ (by \eqref{eq:dv}),
	and the vertical length of the edge from $\bar\beta(\frac{2}{3})$ to $\bar\beta(0)$ (which is reparametrized $\alpha$) is at most $\pi$.
	Since $\varepsilon < \frac{\pi}{2}$, $\ell_v(\bar\beta) < \pi + 2\varepsilon < 2\pi$,
	and hence $\ell(G(1)) = \ell_v(G(1))  = \ell_v(G(0)) = \ell_v(\bar\beta) < 2\pi$.
	It follows that $G(1)$ is contractible,
	and hence $C(t)$ is contractible for any $t \in [c, c+\delta]$.

      \item If $\pi \circ e_i(c)$ and $\pi \circ e_{i-1}(c)$ are not tangent to $\pi \circ e_j(c)$,
        and if $\pi \circ e_i(c)$ and $\pi \circ e_{i-1}(c)$ are on the same side of $\pi \circ e_j(c)$,
	then the singularity is called a \emph{self-tangency}.
	In this case, $H(c + \delta)$ has two more or two less crossings than
	$H(c - \delta)$ has.
	We can show that the two crossings involved are of the same type $g$ but opposite signs,
	and hence $W_g(H(c + \delta)) = W_g(H(c - \delta))$.

	Without loss of of generality,
	assume that $H(t)$ has two more crossing at $(l_1(t), l'_1(t))$ and $(l_2(t), l'_2(t))$ than
	$H(t')$ has when $c - \delta \le t' < c < t \le c + \delta$.
	(See Figure \ref{fig:r2}.)
  \begin{figure}[h]
      \centering
      \begin{subfigure}[b]{0.15\textwidth}
	\centering
	\includegraphics[width=\textwidth]{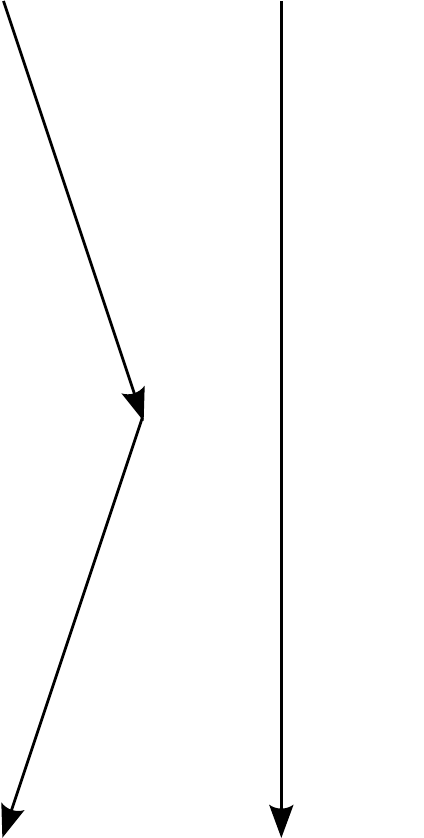}
	\caption{$t = c - \delta$}
	\label{fig:b2_l}
      \end{subfigure}%
      \quad
      \quad
      \quad
      \begin{subfigure}[b]{0.15\textwidth}
	\centering
	\includegraphics[width=\textwidth]{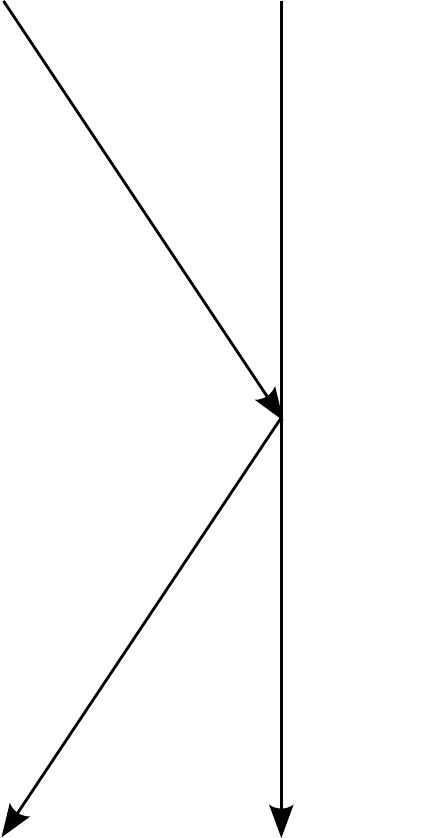}
	\caption{$t = c$}
	\label{fig:b2_0}
      \end{subfigure}
      \quad
      \quad
      \quad
      \begin{subfigure}[b]{0.15\textwidth}
	\centering
	\includegraphics[width=\textwidth]{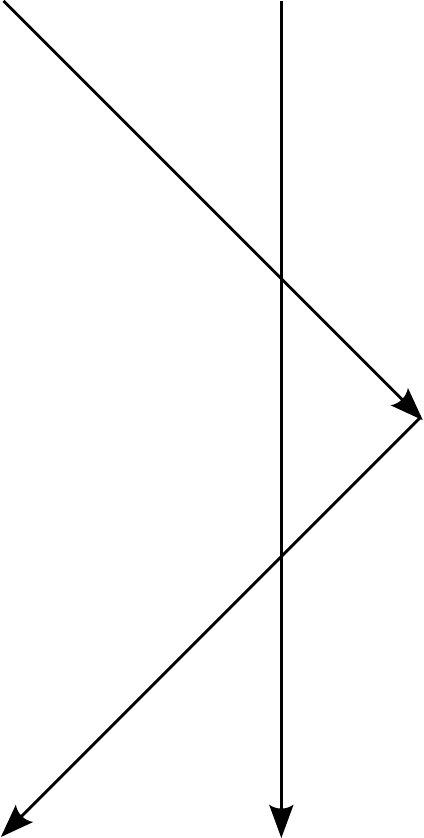}
	\caption{$t = c + \delta$}
	\label{fig:b2_r}
      \end{subfigure}
      \caption{\ref{fig:b2_r} has two more crossing of the same type but opposite signs compared to \ref{fig:b2_l}. This corresponds to a Reidemeister move of type II.}
      \label{fig:r2}
  \end{figure}
	Let $l_1(c) = \lim_{t \rightarrow c+} l_1(t)$ and define $l'_1(c)$, $l_2(c)$ and $l'_2(c)$ similarly.
	Switching $l_2$ and $l'_2$ if necessary,
	we may assume that $l_1(c) = l_2(c)$ and $l'_1(c) = l'_2(c)$.
	Also, either $H(l_1(c)) = x_i$ or $H(l'_1(c)) = x_i$.
	Without loss of generality,
	we assume that $H(l_1(c)) = x_i$.
	Lift $H : [0,1] \rightarrow \mathcal{K}(n, \varepsilon)$ to $\hat{H} : [0, 1] \rightarrow (\sphere \rightarrow \Omega N)$,
	and denote the vertices of $\hat{H}(t)$ by $\hat{x}_k(t)$ and edges by $\hat{e}_k(t)$.

	For any $t \in (c, c+\delta]$,
	we can separate $\hat{H}(t)$ into two arcs by cutting at $\hat{H}(l_1(t))$ and $\hat{H}(l'_1(t))$.
	Pick the arc which contains $\hat{x}_i(t)$ and glue it to $\hat{H}_{(l_1(t), l'_1(t))}$,
	obtaining a closed curve $C_1(t)$.
	We can also separate $\hat{H}(t)$ into two arcs by cutting at $\hat{H}(l_2(t))$ and $\hat{H}(l'_2(t))$.
	Pick the arc which does not contain $i\hat{x}_i(t)$ and glue it to $\hat{H}_{(l_2(t), l'_2(t))}$,
	obtaining a closed curve $C_2(t)$.
	Next, define $C_1(c)$ and $C_2(c)$ by taking limits.
	It is then clear that $P \circ C_1(t)$ and $P \circ C_2(t)$ are in the same unoriented homotopy class since $C_1(c) = C_2(c)$.
	Hence the two crossings at $(l_1(t), l'_1(t))$ and $(l_2(t), l'_2(t))$ are of the same type.

	Finally, it remains to show that the two crossings have opposite signs.
	Without loss of generality, assume that the crossing at $(l_1(t), l'_1(t))$ is positive.
	In other words,
	$(\hat{H}(l_1(t)), \hat{H}(l'_1(t)))$ and $( (\pi \circ H)'(l_1(t)), (\pi \circ H)'(l'_1(t)) )$ have the same orientation.
	It follows that
	$(\hat{H}(l_1(c)), \hat{H}(l'_1(c)))$ and $( \lim_{t \rightarrow c+}(\pi \circ H)'(l_1(t)), (\pi \circ H)'(l'_1(c)) )$ have the same orientation.
	Since $\pi \circ e_i(c)$ and $\pi \circ e_{i-1}(c)$ are on the same side of $\pi \circ e_j(c)$,	$( \lim_{t \rightarrow c+}(\pi \circ H)'(l_1(t)), (\pi \circ H)'(l'_1(c)) )$ and $( \lim_{t \rightarrow c+}(\pi \circ H)'(l_2(t)), (\pi \circ H)'(l'_2(c)) )$ have the opposite orientation.
	Since $(\hat{H}(l_1(c)), \hat{H}(l'_1(c)))$ and $(\hat{H}(l_2(c)), \hat{H}(l'_2(c)))$
	are the same,
	$(\hat{H}(l_2(c)), \hat{H}(l'_2(c)))$ and $( \lim_{t \rightarrow c+}(\pi \circ H)'(l_2(t)), (\pi \circ H)'(l'_2(c)) )$ have the opposite orientation,
	and thus the crossing at $(l_2(t), l'_2(t))$ is negative.

      \item If $\pi \circ e_i$ and $\pi \circ e_{i-1}$ are not tangent to $\pi \circ e_j$,
        and if $\pi \circ e_i$ and $\pi \circ e_{i-1}$ are on different sides of $\pi \circ e_j$,
	then the singularity is called a \emph{transverse self-intersection}.
	In this case, all crossings moves continuously as $t$ goes from $c - \delta$ to $c + \delta$, although one crossing will be also a singularity at $t = c$.
	The type and the sign of that crossing will be unchanged,
	which follows from an argument very similar to the one used for the previous case.
    \end{enumerate}
    In any case,
    we have $W_g(H(c - \delta)) = W_g(H(c + \delta))$ for any non-trivial type $g$.
    It follows that $W_g(\beta_0) = W_g(H(0)) = W_g(H(1)) = W_g(\beta_1)$,
    and thus we may extend $W_g$ to $\mathcal{K}_1(n, \varepsilon)$
    by making it constant on each component.

    Next, we will extent $W_g$ to the whole $\mathcal{K}(n, \varepsilon)$.
    Pick any $\beta_0, \beta_1 \in \mathcal{K}_1(n, \varepsilon)$
    such that $\beta_0, \beta_1$ are in the same component of $\KK$.
    We aim to show that $W_g(\beta_0) = W_g(\beta_1)$,
    and then we can extend $W_g$ to $\mathcal{K}(n, \varepsilon)$
    by making it constant on each component.

    The manifold $\KK$ has a natural stratified structure as follows.
    For any $\gamma \in \KK$,
    pick any neighborhood $U$ of $\gamma$.
    If $\gamma$ has $k$ singularities,
    then let $U_\gamma$ be the component of $U \bigcap \Kp{k}$
    containing $\gamma$, which is a submanifold embedded in $\KK$.
    Now, let $X_m = \{\gamma \in \KK : \dim(U_\gamma) = m\}$.
    Then $\KK$ is a stratified space whose $m$-dimensional stratum is $X_m$.
    We obviously have $X_{3n} = \Kk{0}$ and $X_{3n-1} = \Kp{1}$.
    (Note that $X_{m} = \Kp{3n - m}$ is not true when $m$ is big since the singularities are not necessarily independent. See Figure \ref{fig:stratified}.)
    \begin{figure}[h]
      \centering
      \def\svgwidth{0.4\columnwidth}
      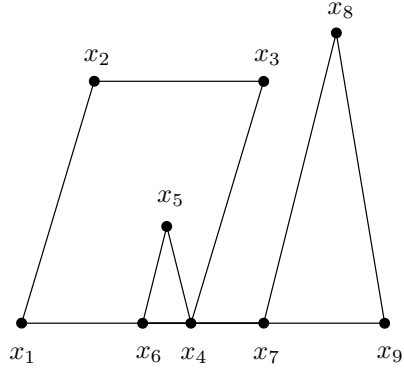
      \caption{This is the projection of a knot $\beta \in X_{24}$ to $N$.
      Notice that $x_4$ is on the edge from $x_9$ to $x_1$ and also the edge from $x_6$ to $x_7$. Hence, there are two singularities involving $x_4$.
      There are also two singularities involving $x_6$ and $x_7$,
      and thus $\beta \in \mathcal{K}'_4(9, \varepsilon)$.
      This counterexample shows that 
      $X_{m} = \Kp{3n - m}$ is not true in general when $m$ is big.
    }
      \label{fig:stratified}
    \end{figure}
    Pick a smooth path $H : [0, 1] \rightarrow \mathcal{K}_1(n, \varepsilon)$ from $\beta_0$ to $\beta_1$.
    Perturbing $H$ if necessary,
    we may assume that $H$ intersects each stratum $X_m$ transversely.
    In other words, $H$ does not intersect $X_m$ at all if $m < 3n-1$.
    Hence $H$ is actually a path in $\Kk{1}$,
    and thus $W_g(\beta_0) = W_g(\beta_1)$.
    Therefore, we can extend $W_g$ to $\KK$
    by making it constant on each component.

    We can extend $W_g$ to a knot invariant for all contractible knots embedded in $P \Omega N$ using Lemma \ref{lem:isotopy}, which will be proved in Appendix \ref{sec:approx}.
    For any smooth contractible knot $\beta$,
    we can approximate $\beta$ by a piecewise linear knot $\beta'$ that is homotopic to $\beta$. Then we set $W_g(\beta) = W_g(\beta')$.
    $W_g$ is well-defined according to Lemma \ref{lem:isotopy}.
  \end{proof}

    Now, we are ready to prove Theorem \ref{thm:knot}.

  \begin{proof}[Proof of Theorem \ref{thm:knot}]
    Let $\beta = P \circ \tilde \gamma$ be the projectivized unit tangent vector field of $\gamma$.

    If $\beta$ is not contractible, then $\beta$ is a non-trivial knot.
    Assume that $\beta$ is contractible. We are going to show that $W_g(\beta) > 0$ for some $g$,
    while $W_g(\text{Unknot}) = 0$ for any $g$.
    
    Assume that $\gamma$ has no self-intersections.
    Then $\beta$ is not contractible by Proposition \ref{prop:embedded}.
    So $\gamma$ has at least one self-intersection.

    We will start at any point on $\gamma$ and trace along $\gamma$
    until hitting the trace.
    To be precise,
    let $q = \max\{t : \text{$\gamma|_{[0, t]}$ has no self-intersection}\}$.
    Then there is $p \in [0, q)$ such that $\gamma(p) = \gamma(q)$
    and $\beta$ has a crossing at $(p, q)$.
    By Lemma \ref{lem:positive},
    the crossing of $\beta$ at $(p, q)$ is positive.
    Separate $\hat\beta$ into two arcs by cutting at $\hat\beta(p)$ and $\hat\beta(q)$.
    Then $\hat{\beta}|_{[p, q]}$ will be one of these two arcs.
    Glue $\hat\beta|_{[p, q]}$ to $\hat\beta_{(p, q)}$,
    obtaining a closed curve $\hat{\beta}'$.
    We can gradually widen the angle of $\gamma|_{[p, q]}$ at the corner
    until it becomes a simple smooth closed curve,
    and $\hat{\beta}'$ will converge to the unit tangent vector field along that simple smooth closed curved.
    By Proposition \ref{prop:embedded}, $\hat{\beta}'$ is not contractible in $\Omega N$.

    Denote by $g$ the non-orientable homotopy type of $\hat\beta'$,
    then $W_g(\beta) \ge 1$ by Lemma \ref{lem:positive}.
    Since $W_g(\text{Unknot}) = 0$,
    $\beta$ is isotopically non-trivial.
  \end{proof}

  Actually, a stronger (but more technical) result can be proved with exactly the same proof.
  \begin{thm}
    Suppose that $\gamma_1 : [0, 1] \rightarrow N$ is smoothly immersed
    curve without self-tangencies
    and that $\beta_2 : [0, 1] \rightarrow P \Omega N$ is a smoothly embedded curve connecting the end points of $\beta_1 := P \circ \tilde{\gamma}$.
    Glue $\beta_2$ to $\beta_1$, obtaining a closed curve $\beta$ in $P \Omega N$.
    If $\gamma_1$ has at least one self-intersection,
    $\gamma_1$ and $\pi \circ \beta_2$ have no intersections
    and $\pi \circ \beta_2$ has no self-intersections,
    then $\beta$ is isotopically non-trivial.
    \label{thm:knot_2}
  \end{thm}
  \begin{proof}
    Just let
    \begin{align*}
      \beta =
      \begin{cases}
	\beta_1(2t) & \text{if $t \in [0, \frac{1}{2}]$},\\
	\beta_2(2 - 2t) & \text{if $t \in [\frac{1}{2}, 1]$}.\\
      \end{cases}
    \end{align*}
    and $\gamma = \pi \circ \beta$.
    Let $q = \max\{t : \text{$\gamma|_{[0, t]}$ has no self-intersection}\}$.
    Then there is $p \in [0, q)$ such that $\gamma(p) = \gamma(q)$
    and $\beta$ has a crossing at $(p, q)$.

    Note that $q < \frac{1}{2}$ because $\gamma_1$ has at least one self-intersection.
    The rest of the proof is very similar to the proof of Theorem \ref{thm:knot}
    since it only involves $\beta_{[p, q]}$.

  \end{proof}

\section{Closed geodesics tangent to the boundary}
\label{proof}
In this section, $M$ and $N$ will be two Riemannian surfaces with the same scattering data rel $h : \partial M \rightarrow \partial N$
where $h$ is an isometry.
Also, $M$ is assumed to be simple.
$\varphi : \partial \Omega M \rightarrow \partial \Omega N$
is the induced bundle map defined in \eqref{eq:phi}.
In this section,
we shall prove Theorem \ref{thm:simple} by studying closed geodesics tangent to the boundary.

Recall that $L = \tau_{N}(\varphi(X)) - \tau_M(X) \ge 0$ is a constant. We need to show that $L = 0$.

For any $Y \in \partial_0 \Omega N$,
Recall that $\gamma_Y$ is the limit of geodesic segments
$\gamma_X$ as $X \rightarrow Y$ where $X \in \partial_+ \Omega N$.
$\gamma_Y$ is a closed geodesic of length $L$.

\begin{prop}
    If $L > 0$,
    then $P \circ \tilde{\gamma}_Y$ is contractible in $P \Omega N$
    for any $Y \in \partial_0 \Omega N$.
    \label{prop:trivial}
\end{prop}
\begin{proof}
  Pick $p \in \partial \Omega N$.
  Let $Y \in \partial_0 \Omega_p N$ be one of the two unit vector at $p$ which are tangent to $\partial N$.
  Put
  \begin{align*}
      Y_s = \cos(\pi s) Y + \sin(\pi s) \nu(x),
  \end{align*}
  for each $s \in [0, 1]$.

  Now define a continuous family of loops $H : [0, 1] \times \sphere \rightarrow \Omega N$ as
  \begin{align*}
    H_s(t) =
    \begin{cases}
      \gamma_{Y_s}'(3 \tau(Y_s) t) & \text{if $0 \le t \le \frac{1}{3}$,}\\
      \alpha_N( Y_{(2 - 3t)s}) & \text{if $\frac{1}{3} \le t \le \frac{2}{3}$,}\\
      Y_{(3t - 2)s} & \text{if $\frac{2}{3} \le t \le 1$.}
    \end{cases}
  \end{align*}

  We shall show that $H_1|_{[\frac{1}{3}, 1]}$ is contractible.
  Since $N$ is a disk,
  there is a diffeomorphism $\psi : N \rightarrow \{(x, y) \in \real^2 : x^2 + y^2 \le 1\}$.
  For any $X \in \Omega N$
  and $x \in N$,
  let $\xi(x, X)$ be the unit vector based at $x$
  such that $\psi_*(\xi(x, X))$ and $\psi_*(X)$
  have the same directions as two vectors in $\real^2$.

  Since $N$ is a disk,
  $\partial N$ is homotopic the constant curve at $p$.
  So there is a homotopy $q : [0, 1] \times [0, 1] \rightarrow N$ such that
  $q(1, \cdot) = p$ and that
  $q(0, t) = \pi(\alpha_N(Y_{t}))$.
  Next, define a continuous family of loops $G : [0, 1] \times \sphere \rightarrow \Omega N$ as
  \begin{align*}
    G_s(t) =
    \begin{cases}
      \xi(q(s, 2t), \alpha_{M}(Y_{2t})) & \text{if $0 \le t \le \frac{1}{2}$,}\\
      Y_{2 - 2t} & \text{if $\frac{1}{2} \le t \le 1$.}
    \end{cases}
  \end{align*}
  Let $A_t$
  be the angle that $r_1(\bar{t}) := \xi(p, \alpha_M(Y_{\bar{t}}))$ rotates by
  as $\bar{t}$ goes from $0$ to $t$.
  We shall show that $A_1 = \pi$.
  Let $B_t$
  be the angle that $r_2(\bar{t}) := \xi(p, \frac{d}{d \bar{t}}q(0, \bar{t}))$ rotates by
  as $\bar{t}$ goes from $0$ to $t$.
  Notice that $\psi(q(0, \bar{t}))$ goes around the unit circle in $\real^2$ for a full circle
  as $t$ goes from $0$ to $1$.
  Hence $B_1 = 2 \pi$.
  Let $C_t$ be the signed angle between $r_1$ and $r_2$.
  Since $r_1(0)$ and $r_2(0)$ have the same direction,
  $C_t = B_t - A_t$.
  For any $t \in (0, 1)$, $\alpha_M(Y_{\bar{t}})$ is not tangent to $\partial N$,
  and hence $C_t \in [0, \pi]$ for $t \in [0, 1]$.
  Since $r_1(1)$ and $r_2(1)$ have opposite directions,
  $C_1 = \pi$,
  which implies that $A_1 = \pi$.
  Therefore, $\xi(p, \alpha_M(Y_{t}))$
  rotates counterclockwise by $\pi$ as $t$ goes from $0$ to $1$.
  On the other hand, the $Y_{2 - 2t}$ 
  rotate by $\pi$ clockwise as $t$ goes from $\frac{1}{2}$ to $1$.
  Hence $G_1$ is contractible.
  It follows that $G_0$ is contractible,
  and thus $H_1|_{[\frac{1}{3}, 1]}$ is contractible.

  Since $H_0|_{[\frac{1}{3}, 1]}$ is constant
  and $H_0|_{[0, \frac{1}{3}]}$ coincides with $\tilde{\gamma}_{Y}$,
  $H_0$ is homotopic to $\tilde{\gamma}_{Y}$.
  Since $H_1|_{[\frac{1}{3}, 1]}$ is contractible 
  and $H_1|_{[0, \frac{1}{3}]}$ coincides with $\tilde{\gamma}_{-Y}$,
  $H_1$ is homotopic to $\tilde{\gamma}_{-Y}$.
  Therefore, $\tilde{\gamma}_{Y}$ is homotopic to $\tilde{\gamma}_{-Y}$.

  If we rotate each vector counterclockwise by $\pi$,
  then $\tilde{\gamma}_{-Y}$ becomes $R(\tilde{\gamma}_{Y})$.
  Hence $\tilde{\gamma}_{Y}$ is homotopic to $R(\tilde{\gamma}_{Y})$.
  It follows that $[\tilde{\gamma}_{Y}] = [R(\tilde{\gamma}_{Y})] = -[\tilde\gamma_Y]$,
  where $[\alpha]$ means the homology class of $\alpha$.
  Since $N$ is a disk, $\Omega N$ is a solid torus,
  and hence $H_1(\Omega N, \integer) = \integer$.
  Thus $[\tilde{\gamma}_{Y}] = 0$,
  that is,
  $\tilde{\gamma}_{Y}$ is contractible.
  Hence $P \circ \tilde\gamma_Y$ is contractible in $P \Omega N$.
\end{proof}

  Fix an orientation of $\partial N$ and let $X_0(x)$
  be the unit vector tangent to $\partial N$ at $x \in \partial N$
  such that $X_0(x)$ and $\partial N$ have the same orientation.
  Let $h_1 : \real / \integer \rightarrow \partial N$
  be an orientation preserving diffeomorphism.
  Pick $\varepsilon > 0$ small and
  let $T : \partial N \rightarrow \partial N$ be a diffeomorphism
  defined as
  \begin{align*}
      T(x) = h_1(h_1^{-1}(x) + \varepsilon),
  \end{align*}
  and let $X_1(x) \in \partial_+ \Omega_x N$ be the vector
  which is tangent to the geodesic from $x$ to $T(x)$.
  Finally, put $X_2(x) = \alpha(X_1(T^{-1}(x)))$.
  When $T^{-1}(x)$, $x$ and $T(x)$ are close,
  both $X_1(x)$ and $X_2(x)$ are close to $X_0(x)$,
  so we may assume that the angle between $X_1(x)$ and $X_2(x)$
  is smaller than $\pi$
  by picking $\varepsilon$ small enough.
  See Figure \ref{YNT}.
  \begin{figure}[h]
      \center
      \includegraphics[width=200pt]{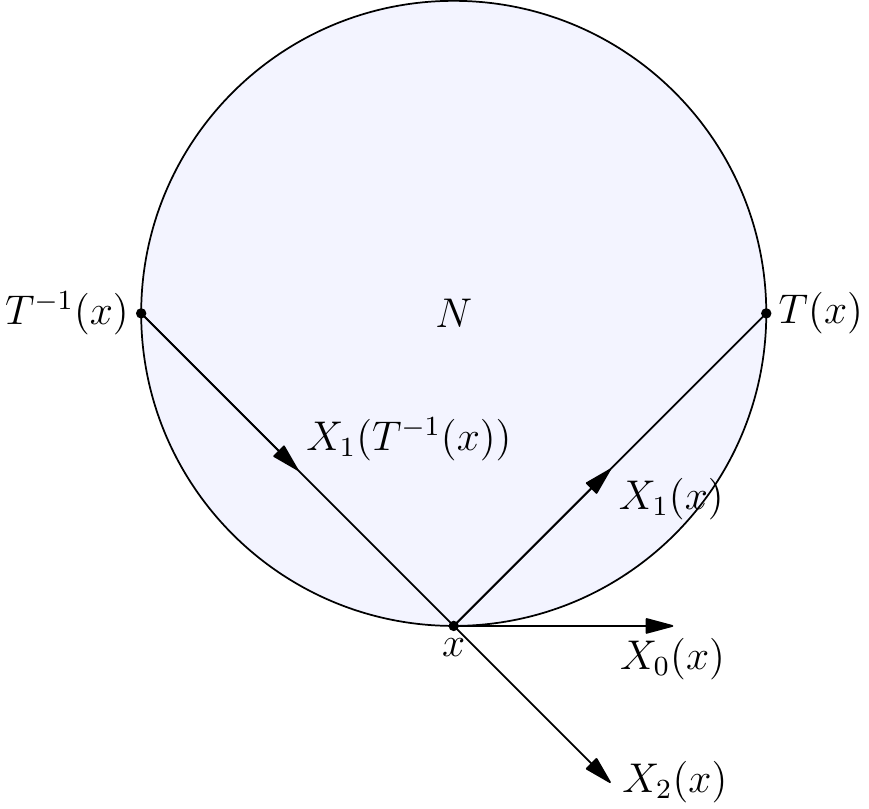}
      \caption{}
      \label{YNT}
  \end{figure}

  Now $X_1(x)$ and $X_2(x)$ separate the circle $\Omega_x N$
  into two segments.
  Let $A(x)$ be the segment containing $X_0(x)$ (which is the shorter segment).
  Then $A = \bigcup_{x \in \partial N} A(x)$ is an annulus with boundaries
  $X_1(\partial N)$ and $X_2(\partial N)$.
  We have a natural diffeomorphism $u : \real / \integer \times [0, 1] \rightarrow A$
  where $u(s, t)$ is the unique vector in $\partial \Omega_{h_1(s)} N$ such that
  \begin{align*}
      t = 
      \frac
      {\text{The angle between $u(x, t)$ and $X_1(h_1(s))$}}
      {\text{The angle between $X_1(h_1(s))$ and $X_2(h_1(s))$}}.
  \end{align*}
  In particular,
  we have $u(s, 0) = X_1(h_1(s))$ and $u(s, 1) = X_2(h_1(s))$.

  Thus we can can break $A$ down to a family of disjoint curves
  $\eta_x : [0, 1] \rightarrow A$
  from $X_1(x)$ to $\alpha(X_1(x))$
  defined as
  \begin{align*}
      \eta_{h_1(s)}(t) = u(s + \varepsilon t, t).
  \end{align*}
  See Figure \ref{fig_eta}.
  \begin{figure}[h!]
      \center
      \includegraphics[width=0.4 \textwidth]{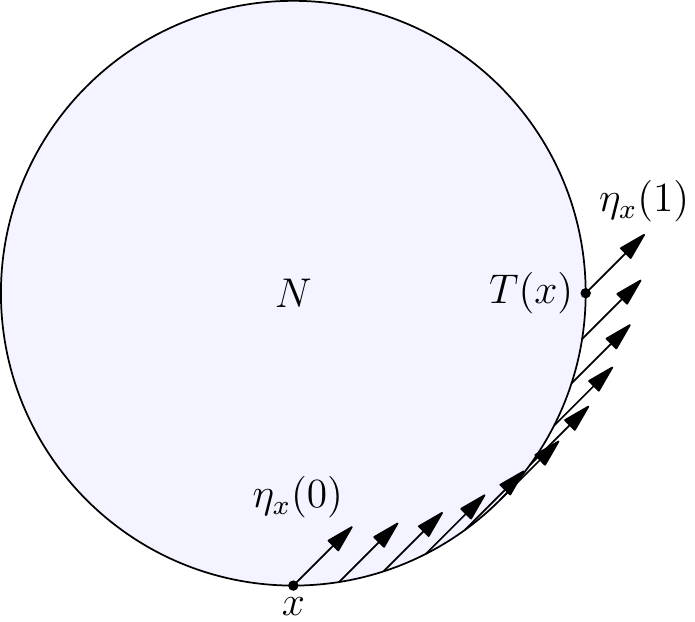}
      \caption{Values of $\eta_x$ from $0$ to $1$}
      \label{fig_eta}
  \end{figure}

  \begin{figure}[h!]
      \center
      \includegraphics[width=0.4 \textwidth]{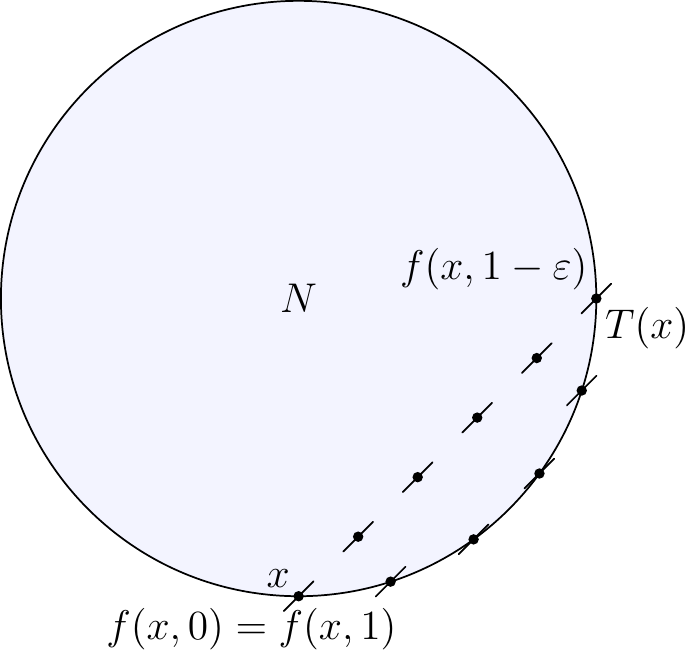}
      \caption{Values of $f(x, \cdot)$ from $0$ to $1$}
      \label{fig_f}
  \end{figure}
  Define
  \begin{align*}
      f : \partial N \times \real / \integer \rightarrow P \Omega N
  \end{align*}
  as
  \begin{align*}
      f(x, t) =
      \begin{cases}
	      P\left(\tilde\gamma_{X_1(x)}(\frac{t}{1 - \varepsilon})\right) & \text{if $0 \le t \le 1 - \varepsilon$,}\\
	      P\left(\eta_x(\frac{1 - t}{\varepsilon})\right) & \text{if $1 - \varepsilon \le t \le 1$.}
      \end{cases}
  \end{align*}
  See Figure \ref{fig_f}.

  \begin{prop}
    $f : \partial N \times \sphere \rightarrow P \Omega N$ is an embedding.
  \end{prop}
  \begin{proof}
  Suppose that $f(x, t) = f(x', t')$.
  If $0 < t < 1 - \varepsilon$,
  then $\pi(f(x, t))$ is not on the boundary,
  so $\pi(f(x', t'))$ is also not on the boundary,
  which implies that $0 < t' < 1 - \varepsilon$.
  However,
  $\pi \circ f(x, \cdot) |_{(0, 1 - \varepsilon)}$
  and 
  $\pi \circ f(x', \cdot) |_{(0, 1 - \varepsilon)}$
  are geodesics in $N$,
  so they always intersect transversely,
  and thus $f(x, t)$ and $f(x', t')$
  are equal if and only if $(x, t) = (x', t')$.
  If $1 - \varepsilon \le t \le 1$,
  then $1 - \varepsilon \le t' \le 1$.
  Now $f(x, t) = P(\eta_x(t))$ and $f(x', t') = P(\eta_{x'}(t'))$,
  where $P : A \rightarrow P \Omega N$ is injective
  because the angle between $X_1(x)$ and $X_2(x)$ is smaller than $\pi$.
  Hence $f(x, t) = f(x', t')$ if and only if $\eta_x(t) = \eta_{x'}(t)$,
  which, by our definition of $\eta$,
  is equivalent to $(x, t) = (x', t')$.
  Therefore, $f$ is an embedding of the torus $\real / \integer \times \partial N$
  into the solid torus $P \Omega N$.
\end{proof}

\begin{prop}
  $f(x, \cdot)$ is contractible in $P \Omega N$.
\end{prop}
\begin{proof}
  We shall show that $f(x, \cdot)$ is homotopic to $P \circ \tilde\gamma_{X_0(x)}$.
  As $\varepsilon \rightarrow 0$,
  $T$ converges to the identity map,
  $X_1$ and $X_2$ converge to $X_0$,
  and $f(x, \frac{t}{1 - \varepsilon})$ converges to $P(\tilde\gamma_{X_0}(t))$ for $t \in [0, 1 - \varepsilon]$.
  Thus $f(x, \cdot)$ is homotopic to $P \circ \tilde\gamma_{X_0(x)}$.
\end{proof}

\begin{prop}
  $f(x, \cdot)$ is isotopically trivial in $P \Omega N$.
  \label{prop:trivial'}
\end{prop}
\begin{proof}
  Define $\beta_1 : \real / \integer \rightarrow \partial N \times \real / \integer$ as
  \begin{align*}
      \beta_1(t) = (h_1(t), 0),
  \end{align*}
  and define $\beta_2 : \real / \integer \rightarrow \partial N \times \real / \integer$ as
  \begin{align*}
      \beta_2(t) = (h_1(0), t).
  \end{align*}
  Let
  $p = \beta_1(0)$,
  then
  \begin{align*}
      \pi_1(\partial N \times \real / \integer, p)
      = \{[\beta_1]_p^k[\beta_2]_p^l : k, l \in \integer\}
      \simeq \integer^2.
  \end{align*}
  Let  
  \begin{align*}
      f_* : \pi_1(\partial N \times \real / \integer, p) \rightarrow \pi_1(P \Omega N, f(p))
  \end{align*}
  be the induced homomorphism between fundamental groups.
  Since $\pi_1(P \Omega N, f(p)) = \integer$,
  $f_*$ is not injective.
  Since \cite[Corollary 3.3]{hatcher2000notes} a two-sided surface $f$ is incompressible if and only if $f_*$ is injective,
  the torus $f(\partial N \times \real / \integer)$ has a compressing disk embedded in $P \Omega N$.

  As $\varepsilon \rightarrow 0$,
  $f \circ \beta_1$ converges to the projectivized unit tangent vector field along $\partial N$,
  so
  $f_*([\beta_1]_p) \neq 0$ by Proposition \ref{prop:embedded}.
  Since $f(h_1(0), \cdot)$ is homotopic to $\gamma_{X_0(h_1(0))}$,
  both of them are contractible in $P \Omega N$ by Proposition \ref{prop:trivial}.
  Since $\pi_1(P \Omega N, f(p)) = \integer$,
  there is no difference between free homotopy and based homotopy,
  hence
  $f_*([\beta_2]_p) = 0$.

  Now, let $B$ be a compressing disk of the torus $f(\partial N \times \real / \integer)$,
  then $\partial B$ is a circle embedded $f(\partial N \times \real / \integer)$ which is contractible in $P \Omega N$.
  It follows that $\partial B$ is homotopic to $f \circ \beta_2$ on $f(\real / \integer \times S)$.
  Any two simple closed essential curves on a surface
  are isotopic to each other if and only if they are freely homotopic to each other \cite{epstein1966curves}.
  Therefore, $\partial B$ is isotopic to $f \circ \beta_2 = f(h_1(0), \cdot)$ on $f(\real / \integer \times S) \subset P \Omega N$.
  Since $\partial B$ bounds a disk $B$ in $P \Omega N$,
  $\partial B$ is isotopically trivial,
  and so is $f(h_1(0), \cdot)$ in $P \Omega N$.
  This completes the proof of Proposition \ref{prop:trivial'}
\end{proof}

Notice (see below) that Proposition \ref{prop:trivial'} contradicts Theorem \ref{thm:knot_2} when $L > 0$, which proves Theorem \ref{thm:simple}.

\begin{proof}[Proof of Theorem \ref{thm:simple}]
  Suppose that $L > 0$.
  Pick any $x \in \partial M$.
  Let $\gamma_1 = \gamma_{X_1(x)}$
  and $\beta_2 = P \circ \eta_x$.

  If $\gamma_1$ has no self-intersections,
  then $\gamma_{X_0(x)}$,
  the limit of $\gamma_1$ as $\varepsilon \rightarrow 0$,
  also has no self-intersections.
  By Proposition \ref{prop:embedded},
  $P \circ \tilde\gamma_{X_0(x)}$
  is not contractible in $P \Omega N_1$,
  which contradicts Proposition \ref{prop:trivial}.
  Therefore, $\gamma_1$ has at least one self-intersection.

  Let $\beta_1 = P \circ \tilde\gamma_1$
  Notice that $\gamma_1$ and $P \circ \beta_2$ have no intersections except at end points and that $f(x, \cdot)$ is the closed curve obtained by gluing $\beta_1$ and $\beta_2$.
  By Theorem \ref{thm:knot_2},
  $f(x, \cdot)$ is isotopically non-trivial in $P \Omega N_1$,
  which contradicts Proposition \ref{prop:trivial'}.
  Thus $L = 0$,
  which finishes the proof of Theorem \ref{thm:simple}.
\end{proof}
\appendix
\section{Approximating smooth knots by piecewise linear knots}
\label{sec:approx}
The goal of this section is to prove the following lemma,
which allows us to approximate knot isotopies using piecewise-linear knot isotopies.
\begin{lem}
  Suppose that there is a continuous knot isotopy $G : [0, 1] \times \sphere \rightarrow P \Omega N$.
  Then there is continuous family of knot isotopies $H : [0, 1] \times [0, 1] \times \sphere \rightarrow P \Omega N$
  such that $H(0, \cdot, \cdot) = G$,
  that $H(l, s, \cdot)$ is a knot embedded in $P \Omega N$ for each $(l, s) \in [0, 1] \times [0, 1]$,
  and that $H(1, s, \cdot)$ is a piecewise linear knot for each $s \in [0, 1]$.
  \label{lem:isotopy}
\end{lem}
The following proposition and its corollaries will be our main tool
used in this section.
\begin{prop}
  Assume that $K$ is a compact smooth manifold and $M$ is a Riemannian manifold.
  Suppose that $G : K \times [a, b] \rightarrow M$ is smooth
  and that each $G(s, \cdot)$ is a smooth curve whose speed is never $0$.
  For any $\varepsilon > 0$,
  there is $\delta > 0$ such that the angles between $G(s, \cdot)|_{[t_1, t_2]}$ 
  and the minimal geodesic connecting its end points are smaller than $\varepsilon$
  whenever $|t_1 - t_2| < \delta$.
  \label{prop:close}
\end{prop}
\begin{proof}
  Pick any $\varepsilon > 0$.
  There is $\varepsilon_1 > 0$ such that
  $|\theta| < \varepsilon$ if $|\cos(\theta) - 1| < \varepsilon_1$
  and if $|\theta| \le \pi$.

  Define $L : K \times [a, b] \times [a, b] \rightarrow \real$ as
  \begin{align*}
    L(s, t_1, t_2) =
    \begin{cases}
      \ell(G(s, \cdot)|_{[t_1, t_2]}) & \text{if $t_2 \ge t_1$,}\\
      -L(s, t_1, t_2) & \text{if $t_2 < t_1$},
    \end{cases}
  \end{align*}
  where $\ell(G(s, \cdot)|_{[t_1, t_2]})$ is the length of $G(s, \cdot)|_{[t_1, t_2]}$.
  Similarly, define
  $D : K \times [a, b] \times [a, b] \rightarrow \real$ as
  \begin{align*}
    D(s, t_1, t_2) =
    \begin{cases}
      d(G(s, t_1), G(s, t_2)) & \text{if $t_2 \ge t_1$,}\\
      -D(s, t_1, t_2) & \text{if $t_2 < t_1$}.
    \end{cases}
  \end{align*}
  For any fixed $s$,
  we have 
  \begin{align}
    L(s, t_1, t_2) - D(s, t_1, t_2) 
    = o((t_2 - t_1)^2)
    \label{eq:close}
  \end{align}
  as $t_2 \rightarrow t_1$.
  Put
  \begin{align*}
    Q(s, t_1, t_2) = \frac{\frac{\partial}{\partial t_2}D(s, t_1, t_2)}{\frac{\partial}{\partial t_2}L(s, t_1, t_2)}.
  \end{align*}
  Then we have $Q(s, t_1, t_1) = 1$ by \eqref{eq:close}.

  We shall show the $Q$ is continuous near $K \times \Delta [a, b]$ where $\Delta[a, b]$ is the diagonal of $[a, b] \times [a, b]$.
  Since $G$ is continuous and $K$ is compact,
  there is $\delta_1 > 0$ such that $d(G(s, t_1), G(s, t_2)) < inj(M)$
  if $|t_1 - t_2| < \delta_1$.
  Since the squared distance function is smooth within the injectivity radius,  $D^2$ is smooth on $K \times V_{\delta_1}$ where $V_{\delta_1} = \{(t_1, t_2) \in [a, b] \times [a, b] : |t_1 - t_2| < \delta_1\}$ and hence $\frac{\partial}{\partial t_2}D$ is continuous.
  Also, $\frac{\partial}{\partial t_2}L$ is obviously continuous on $K \times [a, b] \times [a, b]$ (since $L(s, t_1, \cdot)$ is just the signed arc length).
  Therefore, $Q$ is continuous on $K \times V_{\delta_1}$.
  Since $Q$ is continuous and $K$ is compact,
  there is $\delta \in (0, \delta_1)$ such that $|Q(s, t_1, t_2) - Q(s, t_1, t_1)| < \varepsilon_1$ if $|t_1 - t_2| < \delta$,
  that is,
  $|Q - 1| < \varepsilon_1$ on $K \times V_\delta$.

  For any $(s, t_1, t_2) \in K \times {V_\delta \setminus \Delta[a, b]}$,
  let $\theta(s, t_1, t_2)$ be the angle between
  $G(s, \cdot)|_{[t_1, t_2]}$ and the minimal geodesic connecting its end points
  at the endpoint $G(s, t_2)$.
  By the first variation of arc length,
  \begin{align*}
    \cos\theta(s, t_1, t_2) = \frac{\frac{\partial}{\partial t_2}D(s, t_1, t_2)}{\frac{\partial}{\partial t_2}L(s, t_1, t_2)} = Q(s, t_1, t_2).
  \end{align*}
  Hence $\theta < \varepsilon$ on $K \times V_\delta$.
  Therefore
  the angles between $G(s, \cdot)|_{[t_1, t_2]}$
  and the minimal geodesic connecting its end points are smaller than $\varepsilon$
  whenever $|t_1 - t_2| < \delta$.
\end{proof}

For any compact Riemannian surface $N$,
we can apply the above proposition to unit-speed linear curves of length $\le 1$ in $P \Omega N$ (which has the Sasakian metric on it), which is a compact family of curves in $P \Omega N$.
For any $p, q \in P \Omega N$ such that $d(p, q) < inj(P \Omega N)$ and that $d_h(p, q) < inj(N)$,
let $\gamma_{p, q}$ be the minimal geodesic connecting $p$ and $q$.
Recall that $\gamma^0_{p, q}$ is the minimal linear curve connecting $p$ and $q$.
\begin{cor}
  Assume that $N$ is a compact Riemannian manifold.
  For any $\varepsilon > 0$,
  there is $\delta > 0$ such that the angles between
  $\gamma_{p, q}$ and $\gamma^0_{p, q}$ are smaller than $\varepsilon$
  for any $p, q \in P \Omega N$
  such that $0 < d(p, q) < \delta$.
  \label{cor:linear}
\end{cor}

Suppose that $p, q, r \in P \Omega N$ are close enough that there are minimal linear curves $\gamma^0_{p, q}$, $\gamma^0_{q, r}$ and $\gamma^0_{p, r}$.
Denote by $A(p, q, r)$ the sum of the three angles between $\gamma^0_{p, q}$, $\gamma^0_{q, r}$ and $\gamma^0_{p, r}$.
Since the sum of the inner angles of small geodesic triangles are close to $\pi$, Corollary \ref{cor:linear} implies that $A(p, q, r)$ is also close to $\pi$ when $p, q$ and $r$ are close enough. 
\begin{cor}
  Assume that $N$ is a compact Riemannian manifold.
  For any $\varepsilon > 0$,
  there is $\delta > 0$ such that
  $|A(p, q, r) -  \pi| < \varepsilon$
  for any $p, q, r \in P \Omega N$
  such that $d(p, q), d(p, r), d(q, r) \in (0, \delta)$.
  \label{cor:triangle}
\end{cor}

The following result follows from Proposition \ref{prop:close} and Corollary \ref{cor:linear}.
\begin{cor}
  Assume that $K$ is a compact smooth manifold and $N$ is a compact Riemannian manifold.
  Suppose that $G : K \times [a, b] \rightarrow P \Omega N$ is smooth
  and that each $G(s, \cdot)$ is a smooth curve whose speed is never $0$.
  For any $\varepsilon > 0$,
  there is $\delta > 0$ such that the angles between $G(s, \cdot)|_{[t_1, t_2]}$ and $\gamma^0_{G(s, t_1), G(s, t_2)}$ are smaller than $\varepsilon$
  whenever $|t_1 - t_2| < \delta$.
  \label{cor:linear2}
\end{cor}

\begin{proof}
  [Proof of Lemma \ref{lem:isotopy}]
  We shall assume that $G$ is smooth since it is standard to approximate
  continuous isotopies by smooth isotopies.

  Put $\varepsilon = 0.1$.
  By Corollary \ref{cor:linear2},
  there is $\delta_1 > 0$ such that
  the angles between $G(s, \cdot)|_{[t_1, t_2]}$ and
  $\gamma^0_{G(s, t_1), G(s, t_2)}$ are smaller than $\varepsilon$
  whenever $|t_1 - t_2| < \delta_1$.
  By Corollary \ref{cor:triangle}
  there is $\varepsilon_0 > 0$ such that
  \begin{align}
    |A(p, q, r) -  \pi| < \varepsilon
    \label{eq:pi}
  \end{align}
  for any $p, q, r \in P \Omega N$
  such that $d(p, q), d(p, r), d(q, r) \in (0, \varepsilon_0)$.
  Also, by corollary \ref{cor:linear},
  there is $\varepsilon_1 \in (0, \varepsilon_0)$ such that the angles between
  $\gamma^0_{p, q}$ and $\gamma_{p, q}$ are smaller than $\varepsilon$
  for any $p, q \in P \Omega N$
  such that $0 < d(p, q) < \varepsilon_1$.
  Since $G$ is continuous,
  there is $\delta_2 \in (0, \delta_1)$ such that
  $d(G(s, t_1), G(s, t_2)) < \varepsilon_1$ if $d(t_1, t_2) < \delta_2$.

  For each $(s, t) \in [0, 1] \times \sphere$,
  let $I(t) = \sphere \setminus (t-\delta_2, t+\delta_2)$,
  then $D(s, t) := d(G(s, t), G(s, I(t))) > 0$.
  Let $\varepsilon_2 = \min(\varepsilon_1, \inf_{(s, t) \in [0, 1] \times \sphere} D(s, t))$, then $\varepsilon_2 > 0$.
  Since $G$ is continuous,
  there is $\delta \in (0, \delta_2)$ such that
  $d(G(s, t_1), G(s, t_2)) < \frac{1}{2}\varepsilon_2$ if $d(t_1, t_2) < \delta$.

  Pick $n \in \NN$ such that $n \delta > 1$.
  Define $H : [0, 1] \times [0, 1] \times \sphere \rightarrow P \Omega N$ as
  \begin{align*}
    H(l, s, \frac{k + t}{n}) =
    \begin{cases}
      \gamma^0_{G(s, \frac{k}{n}), G(s, \frac{k+l}{n})}(\frac{t}{l})
      & \text{if $0 \le t < l$}\\
      G(s, \frac{k + t}{n})
      & \text{if $l \le t \le 1$}
    \end{cases}
  \end{align*}
  where $k \in \integer / n \integer$,
  and $t \in [0, 1]$.
  It is obvious that $H(0, \cdot, \cdot) = G$ and $H(1, \cdot, \cdot)$
  is an isotopy of piecewise linear knots.
  We shall show
  that each $H(l, s, \cdot)$ is a knot embedded in $P \Omega N$.

  Suppose that $H(l, s, \cdot)$ has a self-intersection,
  that is,
  $H(l, s, \frac{k_1 + t_1}{n}) = H(l, s, \frac{k_2 + t_2}{n})$
  for some
  $k_i \in \integer /n \integer$ and
  $t_i \in [0, 1)$ such that
  $k_1 \neq k_2$ or $t_1 \neq t_2$.

  Since $G(s, \cdot)$ is an embedding,
  we have either $t_1 < l$ or $t_2 < l$.
  Without loss of generality,
  assume that $t_1 < l$.
  Write $t_3 = \frac{k_1 + t_1}{n}$ and $t_4 = \frac{k_2 + t_2}{n}$.
  We shall show that
  $\gamma^0_{G(s, \frac{k_1}{n}), G(s, \frac{k_1+l}{n})}$
  is in $N_{\frac{\varepsilon_2}{2}}(G(s, \frac{k_1}{n}))$,
  the ball of radius $\frac{\varepsilon_2}{2}$ centered at $G(s, \frac{k_1}{n})$.
  Put $p = G(s, \frac{k_1}{n})$ and $q = G(s, \frac{k_1+l}{n})$.
  Since $d(\frac{k_1}{n}, \frac{k_1+l}{n}) < \delta < \frac{1}{3}\delta_2$,
  $d(p, q) < \varepsilon_1$.
  Define $Q : [0, 1] \rightarrow \real$ as
  \begin{align*}
    Q(t) = \frac{\frac{\partial}{\partial t} d(p, \gamma^0_{p, q}(t))}{\frac{\partial}{\partial t} \ell(\gamma^0_{p, q}|_{[0, t]})}
  \end{align*}
  By the first variation of arc length,
  $Q(t) = \cos \theta(t)$
  where $\theta(t)$ is the angle between $\gamma^0_{p, q}|_{[0, t]}$
  and the minimal geodesic connecting their end points
  at the end point $\gamma^0_{p, q}(t)$.
  Since $d(p, q) < \frac{1}{2}\varepsilon_2 < \varepsilon_1$,
  $\theta(t) < \varepsilon = 0.01$,
  and hence
  $Q(1) > 0$.
  If $Q(t) = 0$ for some $t \in [0, 1)$,
  then let $t_0 = \sup\{t \in [0, 1) : Q(t) = 0\}$.
  Then we have $Q(t_0) = 0$ and $Q(t) > 0$ for $t \in (t_0, 1]$.
  Since
  \begin{align*}
    d(p, \gamma^0_{p, q}(t_0)) &= d(p, q) - \int_{t_0}^1\frac{\partial}{\partial t} d(p, \gamma^0_{p, q}(t))dt\\
    &= d(p, q) - \int_{t_0}^1\frac{\frac{\partial}{\partial t} \ell(\gamma_{p, q}|_{[0, t]})}{Q(t)}dt\\
    &< d(p, q)\\
    &< \varepsilon_1,
  \end{align*}
  $\theta(t_0) < \varepsilon$,
  and hence $Q(t_0) > 0$,
  which contradict our assumption that $Q(t_0) = 0$.
  So $Q(t) > 0$ for any $t \in [0, 1]$,
  which implies that $d(p, \gamma^0_{p, q}(t))$
  is strictly increasing.
  Hence $d(p, \gamma^0_{p, q}(t)) \le d(p, q) < \frac{\varepsilon_2}{2}$ for all $t \in [0, 1)$,
  that is, $\gamma^0_{G(s, \frac{k_1 + l}{n}), G(s, \frac{k_1}{n})}$ is in $N_{\frac{\varepsilon_2}{2}}(p)$.
  It also follows that $\gamma^0_{G(s, \frac{k_1 + l}{n}), G(s, \frac{k_1}{n})}$ has no self-intersections,
  and thus $k_1 \neq k_2$.

  Assume that $d(\frac{k_1}{n}, \frac{k_2}{n}) \ge \delta_2$.
  Then $G(s, \frac{k_2}{n}) \notin N_{{\varepsilon_2}}(G(s, \frac{k_1}{n}))$.
  Hence we have $N_{\frac{\varepsilon_2}{2}}(G(s, \frac{k_1}{n})) \bigcap N_{\frac{\varepsilon_2}{2}}(G(s, \frac{k_2}{n})) = \emptyset$.
  When $t_2 \le l$,
  $H(l ,s, t_4)$ is on $\gamma^0_{G(s, \frac{k_2}{n}), G(s, \frac{k_2 + l}{n})}$,
  which is contained in $N_{\frac{\varepsilon_2}{2}}(G(s, \frac{k_2}{n}))$,
  and hence $H(l ,s, t_4) \in N_{\frac{\varepsilon_2}{2}}(G(s, \frac{k_2}{n}))$.
  However, for the same reason,
  $H(l, s, t_3) \in N_{\frac{\varepsilon_2}{2}}(G(s, \frac{k_1}{n}))$,
  and hence $H(l, s, t_3) \neq H(l, s, t_4)$, which contradicts our assumption.
  When $t_2 \ge l$,
  $H(l, s, t_4) = G(s, t_4)$.
  Since $d(\frac{k_2}{n}, t_4) = \frac{t}{n} < \frac{1}{n} < \delta$,
  $d(G(s, \frac{k_2}{n}), G(s, t_4)) < \frac{\varepsilon_2}{2}$,
  and hence $H(l ,s, t_4) \in N_{\frac{\varepsilon_2}{2}}(G(s, \frac{k_2}{n}))$.
  Again, $H(l, s, t_3) \neq H(l, s, t_4)$, which contradicts our assumption.
  Therefore,
  $d(\frac{k_1}{n}, \frac{k_2}{n}) < \delta_2$.
  We shall assume that $\frac{k_2}{n} \in (\frac{k_1}{n}, \frac{k_1}{n} + \delta_2)$, and the other case ($\frac{k_1}{n} \in (\frac{k_2}{n}, \frac{k_2}{n} + \delta_2)$) can be addressed similarly.
  Then we have $d(H(l, s, \frac{k_1+l}{n}), H(l, s, \frac{k_2}{n})) < \varepsilon_2$.

  We shall show that $d(G(s, \frac{k_2}{n}), H(l, s, t_4)) < \varepsilon_2$.
  If $t_2 \ge l$, then $H(l, s, t_4) = G(s, t_4)$.
  We have $d(G(s, \frac{k_2}{n}), G(s, t_4)) < \varepsilon_2$
  since $d(\frac{k_2}{n}, t_4) < \delta$.
  If $t_2 < l$, then $d(G(s, \frac{k_2}{n}), H(l, s, t_4))
  < d(G(s, \frac{k_2}{n}), H(l, s, \frac{k_2+l}{n})) < \varepsilon_2$.
  Using the same argument, we have
  $d(G(s, \frac{k_1+l}{n}), H(l, s, t_3)) < \varepsilon_2$.

  Write $p_1 = H(l, s, \frac{k_1 + l}{n})$, $p_2 = H(l, s, \frac{k_2}{n})$ and $p_3 = H(l, s, t_3) = H(l, s, t_4)$.
  Then the angle between $\gamma^0_{p_i, p_j}$ and
  $G(s, \cdot)$ is at most $\varepsilon$ for any $i \neq j$.
  Hence the angle between $\gamma^0_{p_1, p_2}$ and $\gamma^0_{p_1, p_3}$
  and the angle between $\gamma^0_{p_1, p_2}$ and $\gamma^0_{p_2, p_3}$
  are at least $\pi - 2\varepsilon$.
  Hence $A(p_1, p_2, p_3) > 2 \pi - 4 \varepsilon$,
  which contradicts \eqref{eq:pi}.

  This completes the proof of Lemma \ref{lem:isotopy}.
\end{proof}

\bibliography{mybib}{}
\bibliographystyle{amsplain}
\end{document}